\documentclass[a4paper,12pt,reqno]{amsart}
\usepackage{amsmath, amssymb,amsthm}
\usepackage{enumerate}
\usepackage{cite}

\nonstopmode \numberwithin{equation}{section}
\newtheorem{theorem}{Theorem}[section]

\newtheorem{corollary}{Corollary}[section]

\newtheorem{remark}{Remark}[section]

\begin{document}
\bibliographystyle{amsplain}

\title{{{ Starlikeness of the generalized integral transform using duality techniques
}}}

\author{
Satwanti Devi
}
\address{
Department of  Mathematics  \\
Indian Institute of Technology, Roorkee-247 667,
Uttarkhand,  India
}
\email{ssatwanti@gmail.com}

\author{
A. Swaminathan
}
\address{
Department of  Mathematics  \\
Indian Institute of Technology, Roorkee-247 667,
Uttarkhand,  India
}
\email{swamifma@iitr.ernet.in, mathswami@gmail.com}

\bigskip

\begin{abstract}
For $\alpha\geq 0$, $\delta>0$, $\beta<1$ and $\gamma\geq 0$,
the class $\mathcal{W}_{\beta}^\delta(\alpha,\gamma)$ consist
of analytic and normalized functions $f$ along with the
condition

\begin{align*}
{\rm Re\,} e^{i\phi}\left(\dfrac{}{}(1\!-\!\alpha\!+\!2\gamma)\!\left({f}/{z}\right)^\delta
+\left(\alpha\!-\!3\gamma\!+\!\gamma\left[\dfrac{}{}\left(1-{1}/{\delta}\right)\left({zf'}/{f}\right)+
{1}/{\delta}\left(1+{zf''}/{f'}\right)\right]\right)\right.\\
\left.\dfrac{}{}\left({f}/{z}\right)^\delta \!\left({zf'}/{f}\right)-\beta\right)>0,
\end{align*}
where $\phi\in\mathbb{R}$ and $|z|<1$, is taken into
consideration. The class $\mathcal{S}^\ast_s(\zeta)$ be the
subclass of the univalent functions, defined by the analytic
characterization ${\rm Re}{\,}\left({zf'}/{f}\right)>\zeta$,
for $0\leq \zeta< 1$, $0<\delta\leq\frac{1}{(1-\zeta)}$ and
$|z|<1$. The admissible and sufficient conditions on
$\lambda(t)$ are examined, so that the generalized and
non-linear integral transforms

\begin{align*}
V_{\lambda}^\delta(f)(z)= \left(\int_0^1 \lambda(t) \left({f(tz)}/{t}\right)^\delta dt\right)^{1/\delta},
\end{align*}
maps the function from
$\mathcal{W}_{\beta}^\delta(\alpha,\gamma)$ into
$\mathcal{S}^\ast_s(\zeta)$. Moreover, several interesting
applications for specific choices of $\lambda(t)$ are
discussed, that are related to some well-known integral
operators.
\end{abstract}

\subjclass[2000]{30C45, 30C55, 30C80}

\keywords{
Duality techniques, Integral transforms, Univalence,
Starlike functions, Convolution, Hypergeometric function.
}

\maketitle

\pagestyle{myheadings} \markboth{Satwanti Devi and A. Swaminathan }{ Starlikeness of the generalized integral transform using duality techniques
}

\section{Introduction}

Let $\mathbb{D}=\{ z\in\mathbb{C}: |z|<1\}$ denote the open
unit disk and $\mathcal{A}$ be the class of all normalized and
analytic functions $f$ defined in the domain $\mathbb{D}$ with
the condition $f(0)\!=\!0\!=f'(0)\!-\!1$. Further, let
$\mathcal{S}\subset\mathcal{A}$ denote the class consisting of
all univalent functions in $\mathbb{D}$.

In \cite{FourRusExtremal}, R. Fournier and S.Ruscheweyh considered the the linear functional
\begin{align*}
L_{\Lambda}(f) =\inf_{z\in{\mathbb{D}}} \int_0^1 \Lambda(t) \left({\rm Re \,} \dfrac{f(tz)}{tz}-\dfrac{1}{(1+t)^2}\right) dt, \quad f\in {\mathcal{S}},
\end{align*}
and
\begin{align*}
L_{\Lambda}(S):=\inf_{f\in {\mathcal{S}}} L_{\Lambda}(f)
\end{align*}
for an integrable function $\Lambda:[0,1]\rightarrow {\mathbb{R}}$, positive in $(0,1)$. Since $L_{\Lambda}(F)\leq 0$ for the Koebe function
$F$, it is clear that $L_{\Lambda}({\mathcal{S}})\leq 0$ for every admissible weight function $\Lambda$ and
R. Fournier and S.Ruscheweyh \cite{FourRusExtremal} posed the problem of finding existence and characterization of the weight functions
for which $L_{\Lambda}({\mathcal{S}}) =0$.  As mentioned in \cite{FourRusExtremal}, since
it is not possible to solve the problem for the class ${\mathcal{S}}$, the main focus was shifted into its important subclass of close-to-convex functions
${\mathcal{C}}$.
For further study of this problem, the following subclass of ${\mathcal{S}}$ is important.

The class $\mathcal{S}^{\ast}(\xi)$ having the analytic characterization
\begin{align*}
{\rm{Re}}{\,} \left(\dfrac{zf^\prime}{f}\right)>\xi,\quad 0\leq\xi<1, \quad z\in\mathbb{D},
\end{align*}
is the generalization of the class of starlike functions, $\mathcal{S}^{\ast}: =\mathcal{S}^{\ast}(0)$, which contains
all the functions $f$ with the property that the domain $f(\mathbb{D})$ is
starlike with respect to the origin.

Using the duality techniques, in \cite{FourRusExtremal}, R. Fournier and S.Ruscheweyh provided solution to two related extremal problems as
\begin{itemize}
\item[(i)] For $\Lambda$ integrable on $[0,1]$ and positive on $(0,1)$, if $\displaystyle \dfrac{\Lambda(t)}{1-t^2}$ is decreasing on $(0,1)$,
then $\displaystyle L_{\Lambda}({\mathcal{C}})=0$, and
\item[(ii)] $\displaystyle V_{\lambda}({\mathcal{P_{\beta}}}) \subset {\mathcal{S}^{\ast}} \Leftrightarrow \displaystyle L_{\Lambda}({\mathcal{C}})=0$,
where
\begin{align}\label{eqn-duality-operator}
V_{\lambda}(f): = \int_0^1 \lambda(t) \dfrac{f(tz)}{t} dt,
\end{align}
with $\lambda:[0,1]\rightarrow {\mathbb{R}}$ is nonnegative, satisfying $\int_0^1 \lambda(t)dt=1$ and
\begin{align*}
{\mathcal{P}}_{\beta} = \left\{ \dfrac{}{} f\in {\mathcal{A}}: {\rm Re\,} \left(\dfrac{}{}e^{i\alpha}(f'(z)-\beta)\right)>0, \quad
\alpha\in {\mathbb{R}}, \quad z\in{\mathbb{D}}\right\}.
\end{align*}
\end{itemize}

Note that the operator \eqref{eqn-duality-operator} contains several well-known operators such as Bernardi, Komatu and Hohlov as its special
cases for specific choices of $\lambda(t)$ and has been studied extensively by several authors. For details on these operators
see \cite{Ali, saigo, DeviPascu} and references therein. Further study of this problem, where the operator \eqref{eqn-duality-operator}
carries generalization of the functional class ${\mathcal{P}}_{\beta}$  involving linear combinations of the functionals
$f(z)/z$, $f'(z)$ and $zf"(z)$ to $\mathcal{S}^{\ast}(\xi)$ were carried by many researchers in the recent future. All these combinations were
unified in a class introduced in \cite{AbeerS} and for the most general result in this direction, see \cite{SarikaStarlike, SuzeiniStarlike, DeviPascu}.
Since such combinations and the generalization of the operator encompass large number of previous results such as univalency, subordination
and positivity results on classes of functions, we are interested in considering the following generalized operator
\begin{align}\label{eq-weighted-integralOperator}
F_\delta(z):= V_{\lambda}^\delta(f)(z)=\left(\int_0^1 \lambda(t) \left(\dfrac{f(tz)}{t}\right)^\delta dt\right)^{1/\delta},
\quad \delta>0.
\end{align}

The integral operator defined in \eqref{eq-weighted-integralOperator} and its more generalized form was considered in the work of I. E. Bazilevi\v c
\cite{BazilevicIntegOper} (see also \cite{Aghalary, SinghIntOperator}). Note that when $\delta=1$, the operator
\eqref{eq-weighted-integralOperator} reduces to \eqref{eqn-duality-operator}. Study of this operator is really useful for the generalization
of the functional class ${\mathcal{P}}_{\beta}$ which is defined as follows.

{\small{
\begin{align*}
\mathcal{W}_{\beta}^\delta(\alpha,\gamma)\!:\!=\!\left\{\!f\!\in\!\mathcal{A}:
{\rm Re\,} e^{i\phi}\left((1\!-\!\alpha\!+\!2\gamma)\!\left(\frac{f}{z}\right)^\delta
\!+\!\left(\alpha\!-\!3\gamma\!+\!\gamma\left[\left(1\!-\!\frac{1}{\delta}\right)\left(\frac{zf'}{f}\right)\!+\!
\frac{1}{\delta}\left(1\!+\!\frac{zf''}{f'}\right)\right]\right)\right.\right.\\
\left.\left.\left(\frac{f}{z}\right)^\delta \!\left(\frac{zf'}{f}\right)-\beta\right)>0,
\,\,z\in\mathbb{D},\,\,\phi\in\mathbb{R}\right\}.
\end{align*}}}
\noindent Here, $\alpha \geq 0$, $\beta<1$, $\gamma\geq 0$ and $\phi\in {\mathbb{R}}$. Note that
$\mathcal{W}_{\beta}^\delta(\alpha,0)\equiv P_\alpha(\delta,\beta)$ is the class considered by A. Ebadian et al in \cite{Aghalary},
$R_\alpha(\delta,\beta):\equiv \mathcal{W}_{\beta}^\delta(\alpha+\delta+\delta\alpha,\delta\alpha)$ is a closely related class and
$\mathcal{W}_{\beta}^1(\alpha,\gamma)\equiv\mathcal{W}_{\beta}(\alpha,\gamma)$ introduced by R.M. Ali et al in \cite{AbeerS}.

We also consider the following class related to $\mathcal{S}^{\ast}(\xi)$, given by
\begin{align}\label{eq-gener:starlike-related:classes}
f\in \mathcal{S}^\ast_s(\zeta)\,\,\Longleftrightarrow\,\,
z^{1-\delta}f^{\delta}\in\mathcal{S}^\ast(\xi),
\end{align}
for $\xi=1-\delta+\delta\zeta$ and $0\leq\xi<1$. It is clear that the class $\mathcal{S}^\ast_s(\zeta)$ has the analytic characterization
\begin{align*}
{\rm{Re}}{\,}\left(\delta\dfrac{zf^\prime}{f}+(1-\delta)\right)>\xi, \quad \delta>0, \quad 0\leq\xi<1.
\end{align*}
Further, when $\delta=1$, $\mathcal{S}^{\ast}_s(\zeta)$ and $\mathcal{S}^\ast(\xi)$ are equal. Wherever $\xi$ is used in
the sequel, it denotes the term $(1-\delta+\delta\zeta)$.

In the present work the duality technique is used to determine the sharp estimate for the parameter $\beta$, so that the
weighted integral operator $V_\lambda^\delta$ defined in \eqref{eq-weighted-integralOperator} carries the function from
$\mathcal{W}_\beta^\delta(\alpha,\gamma)$ to $\mathcal{S}^\ast_s(\zeta)$, where $0\leq \zeta< 1$ and $0<\delta\leq\frac{1}{(1-\zeta)}$.
Certain related preliminaries and the main results involving the necessary and sufficient conditions
are given in Section \ref{Sec-Gener:Starlike-Main:Result} which ensures
$V_\lambda^\delta(\mathcal{W}_\beta^\delta(\alpha,\gamma))\!\subset\!\mathcal{S}^\ast_s(\zeta)$.
The simpler sufficient criterion are obtained in Section \ref{Sec-Gener:Starlike-Suff:Cond}, which verifies
$V_\lambda^\delta(f)(z)\in\mathcal{S}^\ast_s(\zeta)$, whenever $f\in\mathcal{W}_\beta^\delta(\alpha,\gamma)$. Further, using
these sufficient conditions, several interesting applications are studied for specific choices of $\lambda(t)$ are obtained in Section
$\ref{Sec-Gener:Starlike-Appln}$.

A closely related class $\mathcal{C}_\delta(\zeta)$ is defined as
\begin{align*}
\mathcal{C}_\delta(\zeta): = \left\{\dfrac{}{} f\in {\mathcal{A}}:
(z^{2-\delta}f^{\delta-1}f')\in\mathcal{S}^\ast(\xi)\right\},
\end{align*}
where $\xi:=1-\delta+\delta\zeta$ with the conditions
$1-\frac{1}{\delta}\leq\zeta<1$, $0\leq\xi<1$ and
$\delta\geq1$.
In \cite{DeviGenlConvex}, similar analysis for $V_\lambda^\delta(f)(z)\in\mathcal{C}_\delta(\zeta)$,
whenever $f\in\mathcal{W}_\beta^\delta(\alpha,\gamma)$ are given. Various other inclusion properties, in particular,
$V_\lambda^\delta(f)(z)\in\mathcal{W}_{\beta_1}^{\delta_1}(\alpha_1,\gamma_1)$, whenever $f\in\mathcal{W}_{\beta_2}^{\delta_2}(\alpha_2,\gamma_2)$
are given in \cite{DeviGenlUniv}.

\section{Preliminaries and Main results} \label{Sec-Gener:Starlike-Main:Result}

We need the following tools throughout this work.

The convolution or Hadamard product (denoted by `$\ast$'), of
two functions $f_1\!=\!(a_0+a_1z+a_2z^2+\ldots)$ and
$f_2\!=\!(b_0+b_1z+b_2z^2+\ldots)$ is given by

\begin{align*}
(f_1\ast f_2)(z)=\displaystyle\sum_{n=0}^{\infty}a_nb_n
z^n,\quad z\in\mathbb{D}.
\end{align*}

Further, let $c_i$ $(i=0,1, \ldots,p)$ and $d_j$
$(j=0,1,\ldots,q)$ are complex parameters with
$d_j\neq0,-1,\ldots$ and $p\leq q+1$. Then for
$z\in\mathbb{D}$, the function

\begin{align*}
_{p}F_{q}(c_1,\ldots,c_p;d_1,\ldots,d_q;z)
=\sum_{n=0}^\infty
\dfrac{(c_1)_n\ldots,(c_p)_n}{(d_1)_n\ldots,(d_q)_nn!}z^n,
\end{align*}
is called generalized hypergeometric function, which can also
be represented as $_{p}F_{q}$. For $n\in {\mathbb{N}}$,
$(\varepsilon)_n$ is the Pochhammer symbol or shifted
factorial, which is defined as
$(\varepsilon)_n=\varepsilon(\varepsilon+1)_{n-1}$ and
$(\varepsilon)_0=1$. In particular, $_{2}F_{1}$ is the well known
Gaussian hypergeometric function.

The parameters $\mu,\,\nu\geq0$ introduced in \cite{AbeerS} are
used for further analysis that are defined by the following
relations
\begin{align}\label{eq-mu+nu}
\mu\nu=\gamma\quad  \text{and}\quad \mu+\nu=\alpha-\gamma.
\end{align}
Clearly $\eqref{eq-mu+nu}$ leads to two cases.
\begin{itemize}
\item[{\rm{(i)}}] $\gamma=0 \, \Longrightarrow \mu=0,\,
    \nu=\alpha \geq 0$.
\item[{\rm{(ii)}}] $\gamma>0 \, \Longrightarrow \mu>0,\,
    \nu>0$.
\end{itemize}
Define the auxiliary function

\begin{align}\label{eq-generalized:starlike-psi_munu}
\psi_{\mu,\nu}^\delta(z)
:=\sum_{n=0}^\infty\dfrac{\delta^2z^n}{(\delta+n\nu)(\delta+n\mu)}=\int_0^1\int_0^1\dfrac{1}{(1-u^{\nu/\delta/}v^{\mu/\delta}z)}dudv.
\end{align}
Hence

\begin{align}\label{eq-generalized:starlike-Phi_munu}
\Phi_{\mu,\nu}^\delta(z):=\left(z\psi_{\mu,\nu}^\delta(z)\right)'=\sum_{n=0}^\infty\dfrac{(n+1)\delta^2z^n}{(\delta+n\nu)(\delta+n\mu)}
=\dfrac{\delta^2}{\mu\nu}\int_0^1\int_0^1\dfrac{u^{\delta/\nu-1}v^{\delta/\mu-1}}{(1-uvz)^2}dudv.
\end{align}

Taking the case $\gamma=0$ $(\mu=0,\,\nu=\alpha\geq0)$, let
$g_{0,\alpha}^\delta(t)$ be the solution of the differential
equation

\begin{align}\label{eq-generalized:starlike-g:gamma0}
\dfrac{d}{dt}\left(t^{{\delta}/{\alpha}}\dfrac{\left(1+g_{0,\alpha}^\delta(t)\right)}{2}\right)
=\dfrac{\delta(1\!-\!\xi(1\!+\!t))}{\alpha(1\!-\!\xi)(1\!+\!t)^2}t^{{\delta}/{\alpha}-1},
\end{align}
with the initial condition $g_\alpha^\delta(0)=1$. By an easy
calculation, the solution of
\eqref{eq-generalized:starlike-g:gamma0} is given as

\begin{align*}
\dfrac{1+g_{0,\alpha}^\delta(t)}{2}
=\dfrac{\delta t^{-\delta/\alpha}}{\alpha}\int_0^t\dfrac{r^{\delta/\alpha-1}(1-\xi(1+r))}{(1-\xi)(1+r)^2}dr.
\end{align*}

For the case $\gamma>0$ $(\mu>0,\,\nu>0)$, let
$g_{\mu,\nu}^\delta(t)$ be the solution of the differential
equation

\begin{align}\label{eq-generalized:starlike-g:gamma>0}
\dfrac{d}{dt}\left(t^{\delta/\nu}\dfrac{\left(1+g_{\mu,\nu}^\delta(t)\right)}{2}\right)=
\dfrac{\delta^2t^{\delta/\nu-1}}{\mu\nu}\int_0^1\dfrac{1-\xi(1+st)}{(1-\xi)(1+st)^2}s^{\delta/\mu-1}ds,
\end{align}
with the initial condition $g_{\mu,\nu}^\delta(0)=1$. By an
easy calculation, the solution of
\eqref{eq-generalized:starlike-g:gamma>0} can be given as

\begin{align}\label{eq-generalized:starlike-integral:g:gamma>0}
\dfrac{1+g_{\mu,\nu}^\delta(t)}{2}
=\int_0^1\int_0^1\dfrac{1-\xi(1+tr^{\nu/\delta}s^{\mu/\delta})}{(1-\xi)(1+tr^{\nu/\delta}s^{\mu/\delta})^2}drds.
\end{align}
Moreover for given $\lambda(t)$ and $\delta>0$, we introduce

\begin{align}\label{eq-weighted:Lambda}
\Lambda_\nu^\delta(t):=\displaystyle\int_t^1\dfrac{\lambda(s)}{s^{{\delta}/{\nu}}}ds,\quad \nu>0,
\end{align}
and

\begin{align}\label{eq-weighted:Pi}
\Pi_{\mu,\nu}^\delta(t):=
\left\{\!
\begin{array}{cll}\displaystyle \int_t^1\!\!\dfrac{\Lambda_\nu^\delta(s)}{s^{{\delta}/{\mu}-{\delta}/{\nu}+1}}ds
&\quad\gamma\!>\!0 \,(\mu\!>\!0,\nu\!>\!0),\\\\
\Lambda_\alpha^\delta(t)
&\quad\gamma\!=\!0\,(\mu\!=\!0,\nu\!=\!\alpha\!\geq\!0).
\end{array}\right.
\end{align}
These information, for $\delta=1$ coincide with the one given
in \cite{SarikaStarlike}.

Our main aim is to establish both necessary and sufficient
conditions that ensure $F_\delta(z)=V_{\lambda}^\delta(f)(z)\in
\mathcal{S}^\ast_s(\zeta)$, whenever
$f\in\mathcal{W}_{\beta}^\delta(\alpha,\gamma)$.
We state the conditions required for $F_\delta(z)=V_{\lambda}^\delta(\mathcal{W}_{\beta}^\delta(\alpha,\gamma))(z)$
to be in $\mathcal{S}^\ast_s(\zeta)$ and satisfy univalency in the following result and the proof of the
same is given in Section \ref{sec-Gener:Starlike-Main:Result-Proof}.

\begin{theorem}\label{Thm-Gener:Starlike-MainResult}
Let $\mu\!\geq\!0$, $\nu\!\geq\!0$ are defined in
\eqref{eq-mu+nu}, $\delta\geq1$ and
$\left(1-\frac{1}{\delta}\right)\leq\zeta\leq\left(1-\frac{1}{2\delta}\right)$.
Let $\beta\!<\!1$ satisfy

\begin{align}\label{Beta-Cond-Generalized:Starlike}
\dfrac{\beta}{(1-\beta)}=-\int_0^1 \lambda(t) g_{\mu,\nu}^\delta(t) dt,
\end{align}
where $g_{\mu,\nu}^\delta(t)$ is defined by the differential
equation \eqref{eq-generalized:starlike-g:gamma0} for
$\gamma=0$ and \eqref{eq-generalized:starlike-g:gamma>0} for
$\gamma>0$. Assume that

\begin{align*}
\lim_{t\rightarrow 0^+} t^{{\delta}/{\nu}}\Lambda_\nu^\delta(t)\rightarrow 0\quad{\rm and}\quad
\lim_{t\rightarrow 0^+}t^{{\delta}/{\mu}}\Pi_{\mu,\nu}^\delta(t)\rightarrow 0.
\end{align*}
Then for $f(z)\in\mathcal{W}_{\beta}^\delta(\alpha,\gamma)$,
the function $F_\delta\!\in\! \mathcal{S}^\ast_s(\zeta)$ or
$\left(z^{1-\delta}(F_\delta(z))^{\delta}\right)\in\mathcal{S}^\ast(\xi)$,
where $\xi=1-\delta+\delta\zeta$ and $0\leq\xi\leq1/2$ if, and
only if, $\mathcal{N}_{\Pi_{\mu,\nu}^\delta}(h_\xi)\geq 0$,
where

{\small{
\begin{align*}
\mathcal{N}_{\Pi_{\mu,\nu}^\delta}(h_\xi)(z):=\left\{\!\!\!\!
\begin{array}{cll}\displaystyle \int_0^1t^{{\delta}/{\mu}-1}\Pi_{\mu,\nu}^\delta(t) \left(\!{\rm
Re}\dfrac{h_\xi(tz)}{tz}\!-\!\dfrac{1-\xi(1+t)}{(1\!-\!\xi)(1\!+\!t)^2}\right)\!dt ,{\,}
&\gamma>0\,(\mu\!>\!0,\nu\!>\!0),\\\\
\displaystyle\!\!\int_0^1\!\! t^{{\delta}/{\alpha}-1}\Lambda_\alpha^\delta(t)\left(\!{\rm
Re}\dfrac{h_\xi(tz)}{tz}\!-\!\dfrac{1-\xi(1+t)}{(1\!-\!\xi)(1\!+\!t)^2}\right)\!dt,
&\gamma=0\,(\mu\!=\!0,\nu\!=\!\alpha\!\geq\!0),
\end{array}\right.
\end{align*}}}

\begin{align}\label{eq-h(z)-extremal-starlike}
{\rm and}\quad\quad h_\xi(z):=z\left(\dfrac{1+\frac{\epsilon+2\xi-1}{2(1-\xi)}z}{(1-z)^2}\right),\quad |\epsilon|=1.
\end{align}
The value of $\beta$ is sharp.
\end{theorem}

\begin{remark}
\begin{enumerate}[1.]
\item For $\delta=1$ and $\xi=0$, {\rm Theorem
    \ref{Thm-Gener:Starlike-MainResult}} gives the result
    of {\rm\cite[Theorem 3.1]{AbeerS}}.
\item For $\delta=1$, {\rm Theorem
    \ref{Thm-Gener:Starlike-MainResult}} reduces to
    {\rm\cite[Theorem 3.1]{SarikaStarlike}}{\rm (see also
    \cite[Theorem 2.1]{SuzeiniStarlike})}.
\item For $\gamma=0$, {\rm Theorem
    \ref{Thm-Gener:Starlike-MainResult}} provides the
    result of {\rm\cite[Theorem 2.1]{Aghalary}}.
\end{enumerate}
\end{remark}
The condition equivalent to
$\mathcal{N}_{\Pi_{\mu,\nu}^\delta}(h_\xi)\geq 0$ derived in
Theorem \ref{Thm-Gener:Starlike-MainResult} is provided in the
following result which is useful for further discussion.
\begin{theorem}\label{Thm-Gener:starlike-Decr}
Let $\gamma\geq0(\mu\!\geq\!0,\,\nu\!\geq\!0)$, $\delta\geq1$
and
$\left(1\!-\!\frac{1}{\delta}\right)\leq\zeta\leq\left(1\!-\!\frac{1}{2\delta}\right)$.
Assume that the functions $\Lambda_\nu^\delta(t)$ and
$\Pi_{\mu,\nu}^\delta(t)$, defined in
\eqref{eq-weighted:Lambda} and \eqref{eq-weighted:Pi},
respectively are positive on $t\in(0,1)$ and integrable on
$t\in[0,1]$. If $\beta<1$ satisfy
\eqref{Beta-Cond-Generalized:Starlike} and

\begin{align}\label{eq-Gener:starlike-decreas:cond}
\dfrac{t^{\delta/\mu-1}\Pi_{\mu,\nu}^\delta(t)}{(1+t)(1-t)^{3-2\delta(1-\zeta)}}
\end{align}
is decreasing on $t\in(0,1)$. Then for
$f(z)\in\mathcal{W}_\beta^\delta(\alpha,\gamma)$, the function
$F_\delta=V_{\lambda}^\delta(f)(z)\!\in\!
\mathcal{S}^\ast_s(\zeta)$ or
$z^{1-\delta}(F_\delta(z))^{\delta}\in\mathcal{S}^\ast(\xi)$,
where $\xi=1-\delta+\delta\zeta$ and $0\leq\xi\leq1/2$.
\end{theorem}

\begin{proof}
For $t\in(0,1)$, the mapping  $t\rightarrow M(t)$ satisfies the
condition that $M(t)$ is a positive function which decreases
with respect to $t$ and fits into the requirement leading to

\begin{align*}
{\rm Re}\int_0^1M(t)\left(
\dfrac{h_\xi(tz)}{tz}-\dfrac{1-\xi(1+t)}{(1-\xi)(1+t)^2}\right)dt\geq0
\end{align*}
if, and only if,
$F_1(z)=V_\lambda^1(f)(z)\in\mathcal{S}^\ast_s(\zeta)$, where
$0\leq\zeta\leq1/2$ or $z^{1-\delta}(F_\delta(z))^\delta\!\in\!
\mathcal{S}^\ast(\xi)$, where $\xi=1-\delta+\delta\zeta$,
$\delta\geq1$ and $0\leq\xi\leq1/2$, then it is enough to
verify that

\begin{align*}
\dfrac{M(t)}{(1+t)(1-t)^{1+2\xi}}
\end{align*}
is decreasing on $(0,1)$. This is given already in
\cite{SarikaStarlike, SuzeiniStarlike}. Further for $\xi=0$ the
same is proved in \cite{AbeerS, FourRusExtremal}. Hence, if we
take $M(t)=t^{\delta/\mu-1}\Pi_{\mu,\nu}^\delta(t)$, then we
see that \eqref{eq-Gener:starlike-decreas:cond} satisfies the above
observation and hence we get

\begin{align*}
{\rm Re}\int_0^1 t^{\delta/\mu-1}\Pi_{\mu,\nu}^\delta(t)\left(
\dfrac{h_\xi(tz)}{tz}-\dfrac{1-\xi(1+t)}{(1-\xi)(1+t)^2}\right)dt\geq0.
\end{align*}
This proves the required result.
\end{proof}
\begin{remark}
The condition \eqref{eq-Gener:starlike-decreas:cond} obtained in Theorem
\ref{Thm-Gener:starlike-Decr} cannot be reduced to the result
obtained by A. Ebadian et al. \cite[Theorem 2.2]{Aghalary} for
the case $\gamma=0$. This is because the condition given in
\eqref{eq-Gener:starlike-decreas:cond} contains the function
$(1+t)(1-t)^{1+2\xi}$ whereas in \cite{Aghalary}, the result
contains the function $\left(\log(1/t)\right)^{1+2\xi}$ in its
denominator part. Both the functions are decreasing and tends
to $0$ as $t\rightarrow1$.
\end{remark}
\section{Sufficient criterion of Theorem \ref{Thm-Gener:starlike-Decr}
}\label{Sec-Gener:Starlike-Suff:Cond}

In this section, the conditions are determined which ensures
the sufficiency of Theorem \ref{Thm-Gener:starlike-Decr} by a
simpler method, so that the weighted integral operator
$V_\lambda^\delta(\mathcal{W}_\beta^\delta(\alpha,\gamma))\!\subset\!\mathcal{S}^\ast_s(\zeta)$,
with
$\left(1\!-\!\frac{1}{\delta}\right)\leq\zeta\leq\left(1\!-\!\frac{1}{2\delta}\right)$
and $\delta\geq1$.
The conditions comprise of the following two cases.

\textbf{Case (i)}. $\gamma>0$ ($\mu>0$, $\nu>0$). In accordance
of Theorem \ref{Thm-Gener:starlike-Decr}, the equivalent
condition are obtained for the function

\begin{align*}
\dfrac{t^{\delta/\mu-1}\Pi_{\mu,\nu}^\delta(t)}{(1+t)(1-t)^{1+2\xi}},
\end{align*}
which decreases in the range $t\in(0,1)$, where
$\Pi_{\mu,\nu}^\delta(t)$ is defined in \eqref{eq-weighted:Pi},
$\xi=1-\delta+\delta\zeta$, $\delta\geq1$ and $\xi\in[0,1/2]$.

Let

\begin{align*}
k(t):=\dfrac{l(t)}{(1+t)(1-t)^{1+2\xi}}, \quad {\mbox{where}}\quad l(t):=t^{\delta/\mu-1}\Pi_{\mu,\nu}^\delta(t),
\end{align*}
then taking logarithmic derivative of $k(t)$ will give

\begin{align*}
\dfrac{k'(t)}{k(t)}=\dfrac{l'(t)}{l(t)}+\dfrac{2(t+\xi+\xi t)}{(1-t^2)}.
\end{align*}
For $t\in(0,1)$, it is easy to note that $k(t)\geq0$. Thus to
prove that $k(t)$ is decreasing function of $t\in(0,1)$ is
equivalent of getting

\begin{align*}
p(t):=l(t)+\dfrac{(1-t^2)l'(t)}{2(t+\xi(1+t))}\leq0.
\end{align*}
Clearly $p(1)=0$ means that if $p(t)$ is increasing function of
$t\in(0,1)$ then $k'(t)\leq0$ and the proof is complete. Thus
it is enough to show that $p'(t)\geq0$, where

\begin{align*}
p'(t)=\dfrac{(1+t)}{2(t+\xi(1+t))^2}\left(\dfrac{}{}(-1-\xi+2\xi^2+t+3\xi t+2\xi^2t)l'(t)+(1-t)(\xi+t+\xi t)l''(t)\right).
\end{align*}
Differentiating $l(t)$ with respect $t$ gives

\begin{align*}
l'(t)&=\left(\dfrac{\delta}{\mu}-1\right)t^{\delta/\mu-2}\Pi_{\mu,\nu}^\delta(t)-t^{\delta/\nu-2}\Lambda_\nu^\delta(t)\quad\quad{\rm and}\\
l''(t)\!&=\!\left(\dfrac{\delta}{\mu}\!-\!1\right)\!\!\!\left(\dfrac{\delta}{\mu}\!-\!2\right)\!
t^{\frac{\delta}{\mu}-3}\Pi_{\mu,\nu}^\delta(t)\!-\!
\left[\!\left(\dfrac{\delta}{\mu}\!+\!\dfrac{\delta}{\nu}\!-\!3\right)\!\right]\!
t^{\frac{\delta}{\nu}-3}\Lambda_\nu^\delta(t)\!+\!t^{-2}\lambda(t).
\end{align*}
It is easy to see that the terms $(1+t)$ and $2(t+\xi(1+t))^2$
in the function $p'(t)$ are positive for $t\in(0,1)$ and
$\xi\in[0,1/2]$. Now it only remains to show that

\begin{align*}
q(t)=\dfrac{Y(t)}{t}l'(t)+X(t)l''(t)\geq0,
\end{align*}
where

\begin{align}\label{eq-gener:starlike-XY}
X(t):=(1\!-\!t)(t\!+\!\xi(1\!+\!t))\quad{\rm and}\quad Y(t):=-(1\!+\!2\xi)t(1\!-\!t\!-\!\xi(1\!+\!t)).
\end{align}
For $t=1$, the function $q(t)=0$. Therefore, if $q(t)$ is
decreasing function of $t\in(0,1)$ directly implies
$k'(t)\leq0$. Differentiating $q(t)$ with respect to $t$ will
give

{\small{
\begin{align*}
q'(t)=&\left(\!\frac{\delta}{\mu}\!-\!1\!\right)\!t^{{\delta}/{\mu}-4}\Pi_{\mu,\nu}^\delta(t)
\left(\left[tY'(t)+\left(\frac{\delta}{\mu}-3\!\right)Y(t)\right]\!
+\left(\frac{\delta}{\mu}\!-\!2\right)\left[tX'(t)\!+\!\left(\frac{\delta}{\mu}\!-\!3\right)X(t)\right]\right)\\
&-t^{{\delta}/{\nu}-4}\Lambda_\nu^\delta(t)\left(\left(\frac{\delta}{\mu}\!-\!1\right)
\left[Y(t)\!+\!\left(\frac{\delta}{\mu}\!-\!2\right)X(t)\right]\!+\!tY'(t)
\!+\!\left(\frac{\delta}{\mu}\!+\!\dfrac{\delta}{\nu}\!-\!3\right)tX'(t)\right.\\
&\left.+\left(\dfrac{\delta}{\nu}\!-\!3\right)\left[Y(t)\!+\!\left(\dfrac{\delta}{\mu}\!+\!\dfrac{\delta}{\nu}\!-\!3\right)X(t)\right]\right)
+t^{-3}\lambda(t)\left(Y(t)\!+\!\left(\dfrac{\delta}{\mu}\!+\!\dfrac{\delta}{\nu}\!-\!5\right)X(t)\!+\!tX'(t)\right)\\
&+t^{-2}X(t)\lambda'(t).
\end{align*}}}
For $t\in(0,1)$, the function $q'(t)\leq0$ is counterpart of
the following inequalities

\begin{align}\label{eq-Gener:Starlike-suff1(gamma>0)}
\left(\!\frac{\delta}{\mu}\!-\!1\!\right)
\left(\left[tY'(t)+\left(\frac{\delta}{\mu}-3\!\right)Y(t)\right]\!
+\left(\frac{\delta}{\mu}\!-\!2\right)\left[tX'(t)\!+\!\left(\frac{\delta}{\mu}\!-\!3\right)X(t)\right]\right)\leq0,
\end{align}
\begin{align}\label{eq-Gener:Starlike-suff2(gamma>0)}
\left(\frac{\delta}{\mu}-1\right)&\left[Y(t)+\left(\frac{\delta}{\mu}-2\right)X(t)\right]+\left(\dfrac{\delta}{\nu}-3\right)Y(t)+tY'(t)\nonumber\\
&+\left(\frac{\delta}{\mu}+\dfrac{\delta}{\nu}-3\right)\left[\left(\dfrac{\delta}{\nu}-3\right)
X(t)+tX'(t)\right]\geq0
\end{align}
\begin{align*}
{\rm and}\quad\quad\lambda(t)\left(Y(t)\!+\!\left(\dfrac{\delta}{\mu}\!+\!\dfrac{\delta}{\nu}\!-\!5\right)X(t)\!+\!tX'(t)\right)+tX(t)\lambda'(t)\leq0.
\end{align*}
Letting $1\leq\delta\leq\mu$ and $1\leq\delta\leq\nu$, directly
implies that the inequalities
\eqref{eq-Gener:Starlike-suff1(gamma>0)} and
\eqref{eq-Gener:Starlike-suff2(gamma>0)} are positive, which
clearly means that the function $k(t)$ is decreasing on
$t\in(0,1)$ is equivalent of getting

\begin{align}\label{eq-Gener:Starlike-Main:ineq1}
\dfrac{t\lambda'(t)}{\lambda(t)}\leq5-\dfrac{\delta}{\mu}-\dfrac{\delta}{\nu}-\dfrac{(tX'(t)+Y(t))}{X(t)},
\end{align}
where $1\leq\delta\leq\mu$, $1\leq\delta\leq\nu,$ and
$\xi\in[0,1/2]$. Using \eqref{eq-gener:starlike-XY}, an easy
calculation gives $tX'(t)+Y(t)\leq0$, which clearly means that
the inequality \eqref{eq-Gener:Starlike-Main:ineq1} is true
when

\begin{align*}
\dfrac{t\lambda'(t)}{\lambda(t)}\leq5-\dfrac{\delta}{\mu}-\dfrac{\delta}{\nu},\quad1\leq\delta\leq\mu\quad{\rm and}\quad 1\leq\delta\leq\nu.
\end{align*}
Summarizing these conditions, the general result for the case $\gamma>0$ is stated as follows.
\begin{theorem}\label{Thm-gamma>0-Generalized-Starlike}
Let $\beta<1$ satisfy \eqref{Beta-Cond-Generalized:Starlike}
and let $\lambda(t)$ be real-valued, non-negative and
integrable function for $t\in(0,1)$. Further assume that the
functions $\Lambda_\nu^\delta(t)$ and $\Pi_{\mu,\nu}^\delta(t)$
defined in \eqref{eq-weighted:Lambda} and
\eqref{eq-weighted:Pi}, respectively are positive on $(0,1)$
and integrable on $[0,1]$. Then for
$f(z)\in\mathcal{W}_\beta^\delta(\alpha,\gamma)$, the function
$F_\delta=V_\lambda^\delta(f)(z)$ belongs to the class
$\mathcal{S}^\ast_s(\zeta)$ with
$\left(1\!-\!\frac{1}{\delta}\right)\leq\zeta\leq\left(1\!-\!\frac{1}{2\delta}\right)$
or $z^{1-\delta}(F_\delta(z))^\delta\in\mathcal{S}^\ast(\xi)$,
where $\xi=1-\delta+\delta\zeta$, $\xi\in[0,1/2]$,
$\delta\geq1$ and $\gamma>0$ whenever

\begin{align}\label{eq-Gener:Starlike-Main:Cond-gamma>0}
\dfrac{t\lambda'(t)}{\lambda(t)}\leq5-\dfrac{\delta}{\mu}-\dfrac{\delta}{\nu}.
\end{align}
\end{theorem}
{\textbf{Case (ii)}}. Let $\gamma=0$ ($\mu=0$,
$\nu=\alpha\geq0$). In accordance of Theorem
\ref{Thm-Gener:starlike-Decr} the equivalent condition are
obtained for the function

\begin{align}\label{eq-condition-gamma=0}
a(t):=\dfrac{t^{{\delta}/{\alpha}-1}\Lambda_\alpha^\delta(t)}{(1+t)(1-t)^{1+2\xi}}=\dfrac{b(t)}{(1+t)(1-t)^{1+2\xi}},
\end{align}
which decreases in the range $t\in(0,1)$, where
$\xi=(1-\delta+\delta\zeta)$, $\xi\in[0,1/2]$, $\delta\geq1$
and $\Lambda_{\alpha}^\delta(t)$ is defined in
\eqref{eq-weighted:Lambda}.
\newline
For $\gamma=0$, the subsequent two subcases are discussed for
$\xi$.

At first, consider $\xi=0$, then the function $a(t)$
corresponding to \eqref{eq-condition-gamma=0} is given as

\begin{align*}
a(t):=\dfrac{t^{{\delta}/{\alpha}-1}\Lambda_\alpha^\delta(t)}{(1-t^2)}.
\end{align*}
Now, taking the logarithmic derivative of $a(t)$ will give

\begin{align*}
\dfrac{a'(t)}{a(t)}=\dfrac{2t}{(1-t^2)b(t)}\left(b(t)+\dfrac{(1-t^2)b'(t)}{2t}\right).
\end{align*}
Thus to show that $a(t)$ is decreasing function of $t\in(0,1)$
is equivalent of proving $c(t):=b(t)+{(1-t^2)b'(t)}/{2t}\leq0$.
For $t=1$, $c(t)=0$ which clearly implies that if $c'(t)\geq0$
then the function $a(t)$ decreases and the proof is complete.

Now differentiating $c(t)$ gives

\begin{align*}
c'(t)=b'(t)+\dfrac{t(1-t^2)b''(t)-(1+t^2)b'(t)}{2t^2}=\dfrac{(1-t^2)}{2t^2}\left(tb''(t)-b'(t)\right),
\end{align*}
where

\begin{align}\label{eq-gamma=0=b'}
b'(t)=\left(\dfrac{\delta}{\alpha}-1\right)t^{\delta/\alpha-2}\Lambda_\alpha^\delta(t)-t^{-1}\lambda(t)\quad\quad{\rm and}
\end{align}
\begin{align}\label{eq-gamma=0=b''}
b''(t)=\left(\dfrac{\delta}{\alpha}-1\right)\left(\dfrac{\delta}{\alpha}-2\right)t^{\delta/\alpha-3}\Lambda_\alpha^\delta(t)
-\left(\dfrac{\delta}{\alpha}-2\right)t^{-2}\lambda(t)-t^{-1}\lambda'(t).
\end{align}
Thus $c'(t)\geq0$ is equivalent of obtaining

\begin{align*}
\left(\dfrac{\delta}{\alpha}-1\right)\left(\dfrac{\delta}{\alpha}-3\right)\geq0\quad{\rm and}\quad
\dfrac{t\lambda'(t)}{\lambda(t)}\leq\left(3-\dfrac{\delta}{\alpha}\right).
\end{align*}
Thus for $\gamma\geq0$ and $\xi=0$, the following result can be stated.
\begin{theorem}\label{Thm-Gener:Starlike-gamma0:xi0}
Let $\beta<1$ satisfy \eqref{Beta-Cond-Generalized:Starlike}
and let $\lambda(t)$ be real-valued, non-negative and
integrable function for $t\in(0,1)$. Further assume that the
functions $\Lambda_\nu^\delta(t)$ and $\Pi_{\mu,\nu}^\delta(t)$
defined in \eqref{eq-weighted:Lambda} and
\eqref{eq-weighted:Pi}, respectively are positive on $(0,1)$
and integrable on $[0,1]$. Then for
$f(z)\in\mathcal{W}_\beta^\delta(\alpha,\gamma)$, the function
$F_\delta=V_\lambda^\delta(f)(z)$ belongs to the class
$\mathcal{S}^\ast\left(1-\frac{1}{\delta}\right)$ or
$z^{1-\delta}(F_\delta(z))^\delta\in\mathcal{S}^\ast$,
$\delta\geq1$, whenever

\begin{align*}
\dfrac{t\lambda'(t)}{\lambda(t)}\leq
\left\{
  \begin{array}{ll}
    5-\dfrac{\delta}{\mu}-\dfrac{\delta}{\nu}, & \quad\gamma>0(\mu>0,\nu>0)\,\,{\rm and}\,\,1\leq\delta\leq\min\{\mu,\nu\}, \\   \\
    \quad3-\dfrac{\delta}{\alpha}, & \quad\gamma=0{\,\,}\left(\mu=0,{\,}\nu=\alpha\in\left(0,{\delta}/{3}\right]\cup\left[\delta,\infty\right)\right).
  \end{array}
\right.
\end{align*}
\end{theorem}
\begin{remark}
For $\delta=1$, {\rm Theorem
\ref{Thm-Gener:Starlike-gamma0:xi0}} provides better result for
the case $\gamma>0$ and similar result for the case $\gamma=0$
when compared to {\rm\cite[Theorem 4.2]{AbeerS}} and
{\rm\cite[Theorem 4.2]{SarikaStarlike}} {\rm(see also
\cite[Theorem 2.3]{SuzeiniStarlike})}. This is because, for the
case $\delta=1$ and $\gamma>0$

\begin{align*}
\dfrac{t\lambda'(t)}{\lambda(t)}\leq 5-\dfrac{1}{\mu}-\dfrac{1}{\nu},
\end{align*}
which gives $3$ as the least value of right side term of the
above expression. But in {\rm\cite[Theorem 4.2]{AbeerS}}, the
bound is $1+\frac{1}{\mu}$, where $\mu\geq1$, which is clearly
less than or equal to 2.
\end{remark}

\section{Applications}\label{Sec-Gener:Starlike-Appln}
In this section, using the conditions derived in Section $\ref{Sec-Gener:Starlike-Suff:Cond}$,
applications for various choices of $\lambda(t)$ are considered such that the conditions
under which the generalized integral operator \eqref{eq-weighted-integralOperator}, for respective choice maps
$\mathcal{W}_\beta^\delta(\alpha,\gamma)$ to $\mathcal{S}^\ast_s(\zeta)$ are examined.

To start with consider
\begin{align}\label{eq-bernardi-lambda}
\lambda(t)=(1+c)t^c,\quad c>-1,
\end{align}
the integral operator \eqref{eq-weighted-integralOperator}
defined by the above weight function $\lambda(t)$ is known as
generalized Bernardi operator denoted by
$(\mathcal{B}_c^\delta)$. This operator is the particular case
of the generalized integral operators considered by
\cite{Aghalary} that follows in the sequel. For $\delta=1$,
this operator was introduced by S. D. Bernardi \cite{bernardi}.
Now taking this operator the following result is obtained.
\begin{theorem}\label{Thm-Gener:Starlike-Bernardi-gamma>0:gamma0-xi0}
Let $\gamma\geq0$ $(\mu\geq0, \nu\geq0)$, $\xi\in[0,1/2]$,
$\delta\geq1$ and $c>-1$. Further let $\beta<1$ satisfy
\eqref{Beta-Cond-Generalized:Starlike}, where $\lambda(t)$ is
given in \eqref{eq-bernardi-lambda}. Then for
$f(z)\in\mathcal{W}_\beta^\delta(\alpha,\gamma)$, the function
$z^{1-\delta}\left(\mathcal{B}_c^\delta(f)(z)\right)^\delta$
belongs to the class $\mathcal{S}^\ast(\xi)$, whenever

\begin{align*}
c\leq
\left\{
  \begin{array}{ll}
   5-\dfrac{\delta}{\mu}-\dfrac{\delta}{\nu}, & \quad{\rm for}{\,\,\,}\gamma>0{\,\,}
(1\leq\delta\leq\mu,{\,} 1\leq\delta\leq\nu){\,\,\,}{\rm and}{\,\,\,} \xi\in[0,1/2]; \\\\
   \quad3-\dfrac{\delta}{\alpha}, & \quad{\rm for}{\,\,\,}\gamma=0{\,\,}
\left(\mu=0,{\,}\nu=\alpha\in\left(0,{\delta}/{3}\right]\cup\left[\delta,\infty\right)\right){\,\,\,}{\rm and}{\,\,\,} \xi=0.
  \end{array}
\right.
\end{align*}
\end{theorem}
\begin{proof}
Using $\lambda(t)=(1+c)t^c$, $c>-1$, it can be easily seen that
$t\lambda'(t)/\lambda(t)=c$. Applying Theorem
\ref{Thm-gamma>0-Generalized-Starlike} for $\xi\in(0,1/2]$, and
Theorem \ref{Thm-Gener:Starlike-gamma0:xi0} for $\xi=0$, the
result is immediate.
\end{proof}

\begin{remark}
\begin{enumerate}[1.]
\item For $\delta=1$ and $\gamma>0$, {\rm\cite[Theorem
    5.2]{SarikaStarlike}} {\rm(see also \cite[Theorem
    3.1]{SuzeiniStarlike})} give weaker bounds for $c$ when
    compared to {\rm Theorem
    \ref{Thm-Gener:Starlike-Bernardi-gamma>0:gamma0-xi0}}.
\item When $\delta=1$ and $\xi=0$, {\rm Theorem
    \ref{Thm-Gener:Starlike-Bernardi-gamma>0:gamma0-xi0}}
    improves the result given in {\rm\cite[Theorem
    5.1]{AbeerS}} {\rm(see also \cite[Theorem
    5.2]{SarikaStarlike})}.
\end{enumerate}
\end{remark}
Taking $c=0$ in Theorem
\ref{Thm-Gener:Starlike-Bernardi-gamma>0:gamma0-xi0} reduces to
the interesting result, which is listed as a corollary.
\begin{corollary}\label{Corollary-Generalized-Starlike-Bernardi-c=0}
Let $\gamma\geq0(\mu\geq0,\nu\geq0)$, $\xi\in[0,1/2]$ and
$\delta\geq1$. Let $\beta<1$ satisfy

\begin{align*}
\dfrac{\beta}{1-\beta}=-\int_0^1 g_{\mu,\nu}^\delta(t)dt,
\end{align*}
where $g_{\mu,\nu}^\delta(t)$ is given in
\eqref{eq-generalized:starlike-g:gamma>0} for $\gamma>0$, and
\eqref{eq-generalized:starlike-g:gamma0} for $\gamma=0$.
Moreover, $\mathcal{F}(z)\in\mathcal{A}$ satisfies

{\small{
\begin{align}\label{eq-Generalized-Starlike-Bernardi-c=0}
{\rm Re\,} \left(\left[z\left(\!\dfrac{\mathcal{F}(z)}{z}\!\right)^\delta\right]'
\!+\!\dfrac{1}{\delta}\left(\alpha\!-\!\gamma\left(\!1\!-\!\dfrac{1}{\delta}\!\right)\right)
z\left[z\left(\!\dfrac{\mathcal{F}(z)}{z}\!\right)^\delta\right]''
+\dfrac{\gamma}{\delta^2} z^2\left[z\left(\!\dfrac{\mathcal{F}(z)}{z}\!\right)^\delta\right]'''-\beta\right)>0,
\end{align}}}
in the domain $\mathbb{D}$. Then the function
$z^{1-\delta}\mathcal{F}(z)^\delta\in\mathcal{S}^\ast(\xi)$,
whenever

\begin{enumerate}[{\rm(i).}]
  \item $\gamma>0(\mu>0,\nu>0)$,
      $\delta\leq\min\{\mu,\nu\}$ and $\xi\in[0,1/2]$,
  \item $\gamma=0(\mu=0,\nu=\alpha\geq0)$,
      $\alpha\in\left(0,{\delta}/{3}\right]\cup\left[\delta,\infty\right)$
      and $\xi=0$.
\end{enumerate}
\end{corollary}
\begin{proof}
It is apparent that for any function
$\mathcal{F}(z)\in\mathcal{A}$ satisfying the condition
\eqref{eq-Generalized-Starlike-Bernardi-c=0}, then its
corresponding function $f(z)$ defined by the relation
$(f/z)^\delta=(z(\mathcal{F}/z)^\delta)'$ belongs to
$\mathcal{W}_\beta^\delta(\alpha,\gamma)$. Therefore, the
integral representation of $\mathcal{F}(z)$ in terms of $f(z)$
is given by

\begin{align*}
\mathcal{F}(z)=\left(\int_0^1\left(\dfrac{f(tz)}{t}\right)^\delta dt\right)^{1/\delta}.
\end{align*}
With the given hypothesis and taking $c=0$, the result directly
follows from Theorem
\ref{Thm-Gener:Starlike-Bernardi-gamma>0:gamma0-xi0}.
\end{proof}
Using Corollary
\ref{Corollary-Generalized-Starlike-Bernardi-c=0}, the
following two cases are taken into account:
\begin{enumerate}[(i).]
  \item Consider $\gamma=0$, $\delta=1$,
      $\alpha\in\left(0,{1}/{3}\right]\cup\left[1,\infty\right)$
      and $\xi=0$. Let $\beta<1$ satisfy
\begin{align*}
\dfrac{\beta}{1-\beta}=-\int_0^1 g_{0,\alpha}^1(t)dt,
\end{align*}
where $g_{0,\alpha}^1(t)$ is given in
\eqref{eq-generalized:starlike-g:gamma0}. Using {\rm
Corollary
\ref{Corollary-Generalized-Starlike-Bernardi-c=0}}, ${\rm
Re}(\mathcal{F}'(z)\!+\!\alpha
z\mathcal{F}''(z))\!>\!\beta$, implies
$\mathcal{F}(z)\!\in\!\mathcal{S}^\ast$. When $\alpha=1$,

\begin{align*}
{\rm Re}(\mathcal{F}'(z)\!+\! z\mathcal{F}''(z))\!>\beta=\dfrac{1-2\ln 2}{2-2\ln 2}\thickapprox\!-0.62944
\end{align*}
implies ${\mathcal{F}(z)}\in S^*$.

  \item Consider $\alpha=3$, $\gamma=1$, i.e.,
      $(\mu=\nu=1)$, $\xi\in[0,1/2]$ and $\delta=1$. Let
      $\beta<1$ satisfy
\begin{align*}
\dfrac{\beta}{1-\beta}=-\int_0^1 g_{1,1}^1(t)dt,
\end{align*}
where $g_{1,1}^1(t)$ is given in
\eqref{eq-generalized:starlike-g:gamma>0}. Further, using
the series representation of the function $g_{1,1}^1(t)$
given in \eqref{eq-gener:starlike-g:series}, by a simple computation, we get
%
%
      \begin{align*}
\dfrac{\beta}{(1-\beta)}
= 1-\dfrac{\pi^2}{6}.
      \end{align*}
%
%
Thus, using {\rm Corollary
\ref{Corollary-Generalized-Starlike-Bernardi-c=0}},

\begin{align*}
{\rm Re}(\mathcal{F}'(z)+3
z\mathcal{F}''(z)+z^2\mathcal{F}'''(z))>\beta=\dfrac{(\pi^2-6)}{(\pi^2-12)}\thickapprox
-1.81637
\end{align*}
\end{enumerate}
\begin{remark}
The sharp range for $\beta$ improves the result obtained in
{\rm \cite[Example 4.4]{AliAAAThirdOrder}}.
\end{remark}

Secondly, consider the case when $\gamma=0$, $\xi\in[0,1/2]$
and $\lambda(1)=0$. In order to prove Theorem
\ref{Thm-Gener:starlike-Decr} for the given cases, it is enough
to show that the function $a(t)$ defined in
\eqref{eq-condition-gamma=0} decreases for $t\in(0,1)$. Now
taking the logarithmic derivative of $a(t)$ will give

\begin{align*}
\dfrac{a'(t)}{a(t)}=\dfrac{2(t+\xi+\xi t)}{(1-t^2)b(t)}\left(b(t)+\dfrac{(1-t^2)b'(t)}{2(t+\xi+\xi t)}\right).
\end{align*}
It is easy to see that the terms $(t+\xi+\xi t)$, $(1-t^2)$,
$a(t)$ and $b(t)$ are positive for all values of $t\in(0,1)$
and $\xi\in[0,1/2]$. Thus $a'(t)\leq0$ is equivalent of
obtaining $r(t)\leq0$, where

\begin{align*}
r(t):=b(t)+\dfrac{(1-t^2)b'(t)}{2(t+\xi+\xi t)}.
\end{align*}
Clearly $r(1)=0$, therefore if $r(t)$ is increasing function of
$t\in(0,1)$ then $a'(t)\leq0$ which completes the proof. Hence
it is enough to prove

\begin{align*}
r'(t)=\dfrac{(1+t)}{2(t+\xi(1+t))^2}\left(\dfrac{Y(t)}{t}b'(t)+X(t)b''(t)\right)\geq0
\end{align*}
or equivalently,

\begin{align}\label{eq-gamma=0-condition-1}
s(t):=\dfrac{Y(t)}{t}b'(t)+X(t)b''(t)\geq0,
\end{align}
where $X(t)$ and $Y(t)$ are given in
\eqref{eq-gener:starlike-XY}, $b'(t)$ and $b''(t)$ are given in
\eqref{eq-gamma=0=b'} and \eqref{eq-gamma=0=b''}, respectively.
Substituting the value of $b'(t)$ and $b''(t)$ in
\eqref{eq-gamma=0-condition-1}, $s(t)$ is equivalent to

\begin{align*}
s(t)=\left(\dfrac{\delta}{\alpha}\!-\!1\right)\left(Y(t)+\left(\dfrac{\delta}{\alpha}\!-\!2\right)X(t)\right)t^{\delta/\alpha\!-\!3}\Lambda_\alpha^\delta(t)
-\left(Y(t)+\left(\dfrac{\delta}{\alpha}\!-\!2\right)X(t)\right)t^{-2}\lambda(t)\\
-X(t)t^{-1}\lambda'(t).
\end{align*}
The hypothesis $\lambda(1)=0$ directly implies that $s(1)=0$.
If $s(t)$ is decreasing function of $t\in(0,1)$, clearly means
that $a'(t)\leq0$. Differentiating $s(t)$ with respect to $t$
gives

{\small{
\begin{align*}
s'(t)=\left(\dfrac{\delta}{\alpha}-1\right)\left(\left(\dfrac{\delta}{\alpha}-3\right)\left[Y(t)+\left(\dfrac{\delta}{\alpha}-2\right)X(t)\right]
+tY'(t)+\left(\dfrac{\delta}{\alpha}-2\right)tX'(t)\right)t^{\delta/\alpha-4}\Lambda_\alpha^\delta(t)\\
\!+\!\left(\!-\!\left(\!\dfrac{\delta}{\alpha}\!-\!1\!\right)\!\left[\!Y(t)\!+\!
\left(\!\dfrac{\delta}{\alpha}\!-\!2\!\right)\!X(t)\!\right]\!-\!\left[\!tX'(t)
\left(\!\dfrac{\delta}{\alpha}\!-\!2\!\right)\!+\!tY'(t)\!\right]\!+\!2\left[\!X(t)\!\left(\!\dfrac{\delta}{\alpha}\!-\!2\!\right)
\!+\!Y(t)\!\right]\!\right)t^{-3}\lambda(t)\\
+\left(-\left[Y(t)+\left(\dfrac{\delta}{\alpha}-2\right)X(t)\right]-tX'(t)+X(t)\right)t^{-2}\lambda'(t)-X(t)t^{-1}\lambda''(t).
\end{align*}}}
When $1\leq\delta\leq\alpha$, then it is easy to see that

\begin{align*}
\left(\dfrac{\delta}{\alpha}-1\right)\left(\left(\dfrac{\delta}{\alpha}-3\right)
\left[Y(t)+\left(\dfrac{\delta}{\alpha}-2\right)X(t)\right]
+tY'(t)+\left(\dfrac{\delta}{\alpha}-2\right)tX'(t)\right)\leq0.
\end{align*}
Thus, to verify that $a'(t)\leq0$, it is enough to show

{\small{
\begin{align}\label{eq-Gener:Starlike-Main-gamma0:lambda0}
\left(\left(\!\dfrac{\delta}{\alpha}\!-\!1\!\right)\!
\left[\!Y(t)\!+\!\left(\!\dfrac{\delta}{\alpha}\!-\!2\!\right)\!X(t)\!\right]\!+\!\left[\!tX'(t)
\left(\!\dfrac{\delta}{\alpha}\!-\!2\!\right)\!+\!tY'(t)\!\right]\!-\!2\left[\!X(t)\!\left(\!\dfrac{\delta}{\alpha}\!-\!2\!\right)
\!+\!Y(t)\!\right]\!\right)t^{-3}\lambda(t)\nonumber\\
-\left(-\left[Y(t)+\left(\dfrac{\delta}{\alpha}-2\right)X(t)\right]-tX'(t)+X(t)\right)t^{-2}\lambda'(t)+X(t)t^{-1}\lambda''(t)\geq0,
\end{align}}}
for $1\leq\delta\leq\alpha$.

By a simple calculation, the terms

\begin{align}\label{eq-generalized-cond-gamma0}
&\left(\left(\!\dfrac{\delta}{\alpha}\!-\!1\!\right)\!
\left[\!Y(t)\!+\!\left(\!\dfrac{\delta}{\alpha}\!-\!2\!\right)\!X(t)\!\right]\!+\!\left[\!tX'(t)
\left(\!\dfrac{\delta}{\alpha}\!-\!2\!\right)\!+\!tY'(t)\!\right]\!-\!2\left[\!X(t)\!\left(\!\dfrac{\delta}{\alpha}\!-\!2\!\right)
\!+\!Y(t)\right]\right),\nonumber\\
&{\rm and}\quad\quad\left(-\left[Y(t)+\left(\dfrac{\delta}{\alpha}-2\right)X(t)\right]-tX'(t)+X(t)\right)
\end{align}
are positive for $1\leq\delta\leq\alpha$.

Thus, to prove inequality
\eqref{eq-Gener:Starlike-Main-gamma0:lambda0} for
$1\leq\delta\leq\alpha$, it is enough to show
$\lambda(t)\geq0$, $\lambda'(t)\leq0$ and $\lambda''(t)\geq0$.

For the function
\begin{align*}
\omega(1-t)=1+\displaystyle\sum_{n=1}^\infty x_n(1-t)^n,\quad
\quad x_n\geq0,\quad t\in(0,1),
\end{align*}
define
\begin{align}\label{eq-Generalized-hypergeometric-fn}
\lambda(t)=K t^{b-1}(1-t)^{c-a-b}\omega(1-t),
\end{align}
where $K$ is chosen such that it satisfies normalization
condition $\int_0^1\lambda(t)dt=1$. Thus the weighted integral
operator defined in \eqref{eq-weighted-integralOperator} with
$\lambda(t)$ given by \eqref{eq-Generalized-hypergeometric-fn}
is represented as
\begin{align*}
H_{a,b,c}^\delta(f)(z)=\left(K\int_0^1 t^{b-1}(1-t)^{c-a-b}\omega(1-t) \left(\frac{f(tz)}{t}\right)^\delta dt\right)^{1/\delta}.
\end{align*}
This operator is new in the literature whereas for the
particular cases of this operator many interesting results  are
available. For example, when $\delta=1$, the operator
$H_{a,b,c}^1$ was discussed in the literature by several
authors for similar problems. For details refer to
\cite{AbeerS, DeviPascu} and references therein.

 The following
result provides the conditions such that
$(z^{1-\delta}\left(H_{a,b,c}^\delta(f)(z)\right)^\delta)$
belongs to the class $\in\mathcal{S}^\ast(\xi)$.
\begin{theorem}\label{Thm-gamma=>0-Generalized-Starlike-hyper}
Let $\gamma\geq0$ $(\mu\geq0, \nu\geq0)$, $\xi\in[0,1/2]$,
$\delta\geq1$ and $a, b, c>0$. Let $\beta<1$ satisfy
\eqref{Beta-Cond-Generalized:Starlike}, where $\lambda(t)$ is
given by \eqref{eq-Generalized-hypergeometric-fn}. Then for
$f(z)\in\mathcal{W}_\beta^\delta(\alpha,\gamma)$, the function
$(z^{1-\delta}\left(H_{a,b,c}^\delta(f)(z)\right)^\delta)$
belongs to the class $\in\mathcal{S}^\ast(\xi)$, whenever

\begin{enumerate}[{\rm (i).}]
  \item $(c\!-\!a\!-\!b)\geq1$ and $0<
      b\leq1\quad\quad\quad\quad\quad$ for $\gamma=0$ and
      $\delta\leq\alpha$,
  \item $(c\!-\!a\!-\!b)\geq0$ and $0<
      b\leq\left(6\!-\!\dfrac{\delta}{\mu}\!-\!\dfrac{\delta}{\nu}\right)\quad$
      for $\gamma>0$ and $\delta\leq\min\{\mu,\nu\}$.
\end{enumerate}
\end{theorem}
\begin{proof}
Differentiating $\lambda(t)$ defined in
\eqref{eq-Generalized-hypergeometric-fn} will give

\begin{align*}
\lambda'(t)=K t^{b-2}(1-t)^{c-a-b-1}\left(\dfrac{}{}\left[(b-1)(1-t)-(c-a-b)t\right]\omega(1-t)
-t(1-t)\omega'(1-t)\right),
\end{align*}
and

\begin{align*}
&\lambda^{\prime\prime}(t)=Kt^{b-3}(1\!-\!t)^{c\!-\!a\!-\!b\!-\!2}\left[\left(\dfrac{}{}(b\!-\!1)(b\!-\!2)(1\!-\!t)^2-
2(b\!-\!1)(c\!-\!a\!-\!b)t(1\!-\!t)\right.\right.\\
&\left.+(c\!-\!a\!-\!b)(c\!-\!a\!-\!b\!-\!1)t^2\dfrac{}{}\right)\omega(1\!-\!t)+\left[\dfrac{}{}2(c\!-\!a\!-\!b)t-
2(b\!-\!1)(1\!-\!t)\right]t(1\!-\!t)\omega'(1\!-\!t)\quad\quad\\
&\left.\dfrac{}{}+t^2(1-t)^2\omega''(1-t)\right].
\end{align*}
It is easy to note that when $(c-a-b)>0$, then $\lambda(t)$
defined in \eqref{eq-Generalized-hypergeometric-fn} has
$\lambda(1)=0$.

For the case $\gamma=0$ $(\mu=0,\nu=\alpha\geq0)$, from Theorem
\ref{Thm-Gener:starlike-Decr}, we infer that it is enough to
prove \eqref{eq-Gener:Starlike-Main-gamma0:lambda0}. Clearly
for $0\!\leq\!\xi\!\leq\!1/2$ and $1\leq\delta\leq\alpha$, the
conditions in \eqref{eq-generalized-cond-gamma0} are satisfied.
Hence, it remains to check the validity of $\lambda(t)\geq0$,
$\lambda'(t)\leq0$ and $\lambda''(t)\geq0$. However, using the
fact that $\omega(1-t)$, $\omega'(1-t)$ and $\omega''(1-t)$ are
non-negative for $t\in(0,1)$ a simple computation gives
$\lambda(t)\geq0$, $\lambda'(t)\leq0$ and $\lambda''(t)\geq0$,
when $(c-a-b)\geq1$ and $b\leq1$.

Now consider the case $\gamma>0$. Using $\lambda(t)$ given in
\eqref{eq-Generalized-hypergeometric-fn}, we get

\begin{align*}
\dfrac{t\lambda'(t)}{\lambda(t)}=(b-1)-(c-a-b)\dfrac{t}{1-t}-t\dfrac{\omega'(1-t)}{\omega(1-t)}.
\end{align*}
Thus the condition \eqref{eq-Gener:Starlike-Main:Cond-gamma>0}
is true only when

\begin{align}\label{eq-Generalized-Starlike-hypergeometric-cond1-gamma>0}
(b-1)-(c-a-b)\dfrac{t}{1-t}-t\dfrac{\omega'(1-t)}{\omega(1-t)}\leq
\left(5-\frac{\delta}{\mu}\!-\!\frac{\delta}{\nu}\right),
\end{align}
for $1\leq\delta\leq\min\{\mu,\nu\}$. Since $\omega(1-t)$ and
$\omega'(1-t)$ are non-negative on $t\in(0,1)$, therefore
\eqref{eq-Generalized-Starlike-hypergeometric-cond1-gamma>0} is
satisfied if

\begin{align*}
(b-1)-(c-a-b)\dfrac{t}{1-t}\leq\left(5-\frac{\delta}{\mu}\!-\!\frac{\delta}{\nu}\right),
\end{align*}
which is true whenever $c\geq a+b$ and
$b\leq\left(6-\frac{\delta}{\mu}\!-\!\frac{\delta}{\nu}\right)$.
Thus, by the given hypothesis and Theorem
\ref{Thm-Gener:starlike-Decr}, the result follows directly.
\end{proof}
\begin{remark}
\begin{enumerate}[1.]
\item When $\delta=1$ and $\gamma>0$, then {\rm Theorem
    \ref{Thm-gamma=>0-Generalized-Starlike-hyper}} gives
    better range for $b$ when compared to \cite[Theorem
    5.1]{SarikaStarlike} {\rm(see also \cite[Theorem
    5.5]{AbeerS} for the case $\xi=0$)}.
\item When $\delta=1$ and $\gamma=0$, {\rm Theorem
    \ref{Thm-gamma=>0-Generalized-Starlike-hyper}} cannot
    be compared with {\rm\cite[Theorem
    5.1]{SarikaStarlike}}. This is due to the fact that the
    bound for $\alpha$ is different in both the cases.
    Also, when $\xi=0$, due to different range for
    $\alpha$, the bounds for $a$, $b$ and $c$ are different
    in {\rm Theorem
    \ref{Thm-gamma=>0-Generalized-Starlike-hyper}} when
    compared with {\rm\cite[Theorem 2.4]{KimRonning}}.
\end{enumerate}
\end{remark}
Consider
\begin{align}\label{eq-komatu_operator}
\lambda(t)=\dfrac{(1+k)^p}{\Gamma(p)}t^{k}\left(\log\dfrac{1}{t}\right)^{p-1},
\quad \delta\geq 0\quad k>-1.
\end{align}
Then the integral operator \eqref{eq-weighted-integralOperator}
defined by the above weight function $\lambda(t)$ is the known
as generalized Komatu operator denoted by $(F_{k,\,p}^\delta)$.
This integral operator was considered in the work of A. Ebadian
\cite{Aghalary}. When $\delta=1$, the operator is reduced to
the one introduced by Y. Komatu \cite{komatu}.

Now, we are in the position to state the following result.
\begin{theorem}\label{Thm-Generalized-Starlike-Komatu}
Let $\gamma\geq0$ $(\mu\geq,\nu\geq0)$, $k>-1$, $p\geq1$,
$\xi\in[0,1/2]$ and $\delta\geq1$. Let $\beta\!<\!1$ satisfy
\eqref{Beta-Cond-Generalized:Starlike}, where $\lambda(t)$ is
given in \eqref{eq-komatu_operator}. Then for
$f(z)\in\mathcal{W}_\beta^\delta(\alpha,\gamma)$, the function
$z^{1-\delta}\left(F_{k,p}^\delta(f)(z)\right)^\delta$ belongs
to the class $\in\mathcal{S}^\ast(\xi)$, whenever

\begin{enumerate}[{\rm(i).}]
  \item $p\geq2\quad$ and $\quad
      -1<k\leq0\quad\quad\quad\quad\quad$ for $\gamma=0$
      and $\delta\leq\alpha$,
  \item $p\geq1\quad$ and $\quad
      -1<k\leq5-\dfrac{\delta}{\mu}-\dfrac{\delta}{\nu}\quad$
      for $\gamma>0$ and $\delta\leq\min\{\mu,\nu\}$.
\end{enumerate}
\end{theorem}
\begin{proof}
Letting $(c-a-b)=p-1$, $b=k+1$ and
$\omega(1-t)=\left(\frac{\log(1/t)}{(1-t)}\right)^{p-1}$.
Therefore $\lambda(t)$ given in
\eqref{eq-Generalized-hypergeometric-fn} can be represented as

\begin{align*}
\lambda(t)=Kt^{k}(1-t)^{p-1}\omega(1-t),\quad{\mbox{where}\quad}  K=\dfrac{(1+k)^p}{\Gamma(p)}.
\end{align*}
Now, by the given hypothesis the result directly follows from
Theorem \ref{Thm-gamma=>0-Generalized-Starlike-hyper}.
\end{proof}
\begin{remark}
\begin{enumerate}[1.]
\item For $\delta=1$ and $\gamma>0$, {\rm Theorem
    \ref{Thm-Generalized-Starlike-Komatu}} yield better
    range for $k$ when compared to \cite[Theorem
    5.4]{SarikaStarlike} {\rm(see also \cite[Theorem
    5.4]{AbeerS} for the case $\xi=0$)}.
\item For $\delta=1$ and $\gamma=0$, {\rm Theorem
    \ref{Thm-Generalized-Starlike-Komatu}} cannot be
    compared with {\rm\cite[Theorem 5.4]{SarikaStarlike}}.
    This is due to the fact that the bound for $\alpha$ is
    different in both the cases.
\end{enumerate}
\end{remark}
Let
\begin{align*}
\lambda(t)=\dfrac{\Gamma(c)}{\Gamma(a)\Gamma(b)\Gamma(c-a-b+1)}
t^{b-1}(1-t)^{c-a-b}{\,}_{2}F_1\left(\!\!\!\!
\begin{array}{cll}&\displaystyle c-a,\quad 1-a
\\
&\displaystyle c-a-b+1
\end{array};1-t\right),
\end{align*}
then the integral operator \eqref{eq-weighted-integralOperator}
defined by the above weight function $\lambda(t)$ is the known
as generalized Hohlov operator denoted by
$\mathcal{H}_{a,b,c}^\delta$. This integral operator was
considered in the work of A. Ebadian \cite{Aghalary}. When
$\delta=1$, the reduced integral transform can be represented
as the convolution of the normalized hypergeometric function
with the analytic function $z{\,}_2F_1(a,b;c;z)\ast f(z)$,
which was introduced in the work of Y. C. Kim and F. Ronning
\cite{KimRonning} and studied by several authors later. The
operator $\mathcal{H}_{a,b,c}^\delta$ with $a=1$ is the
generalized Carlson-Shaffer operator
($\mathcal{L}_{b,c}^\delta$) \cite{CarlsonShaffer}.

Using the above operators the following results are obtained.

\begin{theorem}\label{Thm-Generalized-Starlike-Hohlov}
Let $\gamma\geq0$ $(\mu\geq0, \nu\geq0)$, $\delta\geq1$,
$\xi\in[0,1/2]$ and $a, b, c>0$. Let $\beta\!<\!1$ satisfy

{\small{
\begin{align}\label{eq-Generalized-Starlike-beta-Hohlov}
\dfrac{\beta}{1\!-\!\beta}=-\dfrac{\Gamma(c)}{\Gamma(a)\Gamma(b)\Gamma(c\!-\!a\!-\!b\!+\!1)}\int_0^1\!\! t^{b-1}(1\!-\!t)^{c-a-b}
{\,}_{2}F_1\left(\!\!\!\!\!\!
\begin{array}{cll}&\displaystyle c\!-\!a,{\,} 1\!-\!a
\\
&\displaystyle c\!-\!a\!-\!b\!+\!1
\end{array};1\!-\!t\right)
g_{\mu,\nu}^\delta(t)dt,
\end{align}}}
where $g_{\mu,\nu}^\delta(t)$ is given in
\eqref{eq-generalized:starlike-g:gamma>0} for $\gamma>0$, and
\eqref{eq-generalized:starlike-g:gamma0} for $\gamma=0$. Then
for $f(z)\in\mathcal{W}_\beta^\delta(\alpha,\gamma)$, the
function
$z^{1-\delta}\left(\mathcal{H}_{a,b,c}^\delta(f)(z)\right)^\delta$
belongs to the class $\mathcal{S}^\ast(\xi)$, whenever
\begin{enumerate}[{\rm (i).}]
  \item $(c\!-\!a\!-\!b)\geq1\quad$ and $\quad0<
      b\leq1\quad\quad\quad\quad\quad$ for $\gamma=0$ and
      $\delta\leq\alpha$,
  \item $(c\!-\!a\!-\!b)\geq0\quad$ and $\quad0<
      b\leq6-\dfrac{\delta}{\mu}-\dfrac{\delta}{\nu}\quad$
      for $\gamma>0$ and $\delta\leq\min\{\mu,\nu\}$.
\end{enumerate}
\end{theorem}
\begin{proof}
Choosing

\begin{align*}
K=\dfrac{\Gamma(c)}{\Gamma(a)\Gamma(b)\Gamma(c\!-\!a\!-\!b\!+\!1)}\quad{\rm and}\quad
\omega(1-t)=
{\,}_{2}F_1\left(\!\!\!\!\!
\begin{array}{cll}&\displaystyle c-a,{\,} 1-a
\\
&\displaystyle c-a-b+1
\end{array};1-t\right),
\end{align*}
in Theorem \ref{Thm-gamma=>0-Generalized-Starlike-hyper} will
give the required result.
\end{proof}
For $a=1$, Theorem \ref{Thm-Generalized-Starlike-Hohlov} lead
to the following corollaries.
\begin{corollary}
Let $\gamma\geq0$ $(\mu\geq0, \nu\geq0)$, $\delta\geq1$,
$\xi\in[0,1/2]$ and $b, c>0$. Let $\beta<1$ satisfy

\begin{align*}
\dfrac{\beta}{(1-\beta)}=-\dfrac{\Gamma(c)}{\Gamma(b)\Gamma(c-b)}\int_0^1
t^{b-1}(1-t)^{c-b-1}g_{\mu,\nu}^\delta(t)dt,
\end{align*}
where $g_{\mu,\nu}^\delta(t)$ is given in
\eqref{eq-generalized:starlike-g:gamma>0} for $\gamma>0$, and
\eqref{eq-generalized:starlike-g:gamma0} for $\gamma=0$. Then
for $f(z)\in\mathcal{W}_\beta^\delta(\alpha,\gamma)$, the
function
$z^{1-\delta}\left(\mathcal{L}_{b,c}^\delta(f)(z)\right)^\delta$
belongs to the class $\mathcal{S}^\ast(\xi)$, whenever
\begin{enumerate}[{\rm (i).}]
  \item $(c\!-\!b)\geq2\quad$ and $\quad0<
      b\leq1\quad\quad\quad\quad\quad$ for $\gamma=0$ and
      $\delta\leq\alpha$,
  \item $(c\!-\!b)\geq1\quad$ and $\quad0<
      b\leq6-\dfrac{\delta}{\mu}-\dfrac{\delta}{\nu}\quad$
      for $\gamma>0$ and $\delta\leq\min\{\mu,\nu\}$.
\end{enumerate}
\end{corollary}
\begin{corollary}
Consider $\gamma\geq0$ $(\mu\geq0,\nu\geq0){\,}$, $b>0$, $c>0$
and $\delta\geq1$. Let $\beta_0<\beta<1$, where
\begin{align*}
\beta_0=1-\dfrac{1}{2\left(1-{\,}
{\,}_5F_4
\left(
\begin{array}{cll}&\displaystyle \quad\quad 1,b,(2-\xi),\dfrac{\delta}{\mu},\dfrac{\delta}{\nu},
\\
&\displaystyle c,(1-\xi),\left(1+\dfrac{\delta}{\mu}\right),\left(1+\dfrac{\delta}{\nu}\right)
\end{array}{\,};{\,}-1\right)
\right)}.
\end{align*}
Then, for $f\in\mathcal{W}_{\beta}^\delta(\alpha,\gamma)$, the
function
$\mathcal{L}_{b,c}^\delta(f)(z)\in\mathcal{S}^\ast_s(\zeta)$ or
$(z^{1-\delta}\left(\mathcal{L}_{b,c}^\delta(f)(z)\right)^\delta)\in\mathcal{S}^\ast(\xi)$,
where $\xi=(1-\delta(1-\zeta))$, $\xi\in[0,1/2]$, whenever
\begin{enumerate}[{\rm (i).}]
  \item $(c\!-\!b)\geq2\quad$ and $\quad0<
      b\leq1\quad\quad\quad\quad\quad$ for $\gamma=0$ and
      $\delta\leq\alpha$,
  \item $(c\!-\!b)\geq1\quad$ and $\quad0<
      b\leq6-\dfrac{\delta}{\mu}-\dfrac{\delta}{\nu}\quad$
      for $\gamma>0$ and $\delta\leq\min\{\mu,\nu\}$.
\end{enumerate}
\end{corollary}
\begin{proof}
Letting $a=1$ in \eqref{eq-Generalized-Starlike-beta-Hohlov}
and using \eqref{eq-gener:starlike-g:hyper:series} gives
{\small{
\begin{align*}
\dfrac{\beta}{1\!-\!\beta}=-\dfrac{\Gamma(c)}{\Gamma{(b)}\Gamma{(c\!-\!b)}}\int_0^1\!\!t^{b-1}(1\!-\!t)^{c-b-1}
\left(\!\! 2{\,}_4F_3\!
\left(\!\!\!\!\!\!\!\!
\begin{array}{cll}&\displaystyle \quad\quad 1,(2-\xi),\dfrac{\delta}{\mu},\dfrac{\delta}{\nu}
\\
&\displaystyle (1\!-\!\xi),\left(1\!+\!\dfrac{\delta}{\mu}\right),\left(1\!+\!\dfrac{\delta}{\nu}\right)
\end{array}{\,};{\,}-t\right) -1\right) dt,
\end{align*}}}
which is equivalent to {\small{
\begin{align*}
\dfrac{\beta-1/2}{1-\beta}=-\dfrac{\Gamma(c)}{\Gamma{(b)}\Gamma{(c-b)}}\int_0^1t^{b-1}
(1-t)^{c-b-1}{\,}_4F_3\!
\left(\!\!\!\!\!\!\!\!
\begin{array}{cll}&\displaystyle \quad\quad 1,(2-\xi),\dfrac{\delta}{\mu},\dfrac{\delta}{\nu}
\\
&\displaystyle (1\!-\!\xi),\left(1\!+\!\dfrac{\delta}{\mu}\right),\left(1\!+\!\dfrac{\delta}{\nu}\right)
\end{array}{\,};{\,}-t\right) dt.
\end{align*}}}
Using the series representation of generalized hypergeometric function and applying integration over $t\in (0,1)$ provides
\begin{align*}
\dfrac{\beta-1/2}{1-\beta}=
%
%
=-{\,}_5F_4\!
\left(\begin{array}{cll}&\displaystyle \quad\quad 1,b,(2-\xi),\dfrac{\delta}{\mu},\dfrac{\delta}{\nu}
\\
&\displaystyle c,(1-\xi),\left(1+\dfrac{\delta}{\mu}\right),\left(1+\dfrac{\delta}{\nu}\right)
\end{array}{\,};{\,}-1\right).
\end{align*}
Thus, applying Theorem \ref{Thm-Generalized-Starlike-Hohlov}
will give the required result.
\end{proof}

For the two complex parameters $a,b>-1$, consider
\begin{align}\label{ponn-oper}
\lambda(t)= \begin{cases}
(a+1)(b+1)\dfrac{t^a(1-t^{b-a})}{b-a},
&b\neq a,\\
(a+1)^2t^a\log(1/t),  &  b=a.
\end{cases}
\end{align}
Then the corresponding integral operator
\eqref{eq-weighted-integralOperator} obtained using
$\lambda(t)$ defined in \eqref{ponn-oper} is denoted by
$\mathcal{G}_{a,b}^\delta$. This operator is new in the
literature whereas for the particular cases of this operator
many interesting results  are available. For example, when
$\delta=1$, the operator was considered by several authors
(see, for example \cite{AbeerS, DeviPascu, SarikaStarlike}).

Now, using this generalized operator the next result is given.

\begin{theorem}\label{Thm-Generalized-starlike-pons}
Let $a>-1$, $b>-1$, $\gamma\geq0$ $(\mu\geq0, \nu\geq0)$,
$\xi\in[0,1/2]$ and $\delta\geq1$. Let $\beta<1$ satisfy
\eqref{Beta-Cond-Generalized:Starlike}, where $\lambda(t)$ is
given in \eqref{ponn-oper}. Then for
$f(z)\in\mathcal{W}_\beta^\delta(\alpha,\gamma)$, the function
$z^{1-\delta}\left(\mathcal{G}_{a,b}^\delta(f)(z)\right)^\delta$
belongs to the class $\mathcal{S}^\ast(\xi)$, whenever

\begin{align*}
-1<b=a\leq\left\{
            \begin{array}{ll}
              5-\dfrac{\delta}{\mu}-\dfrac{\delta}{\nu}, & \quad\gamma>0{\,}(\mu>0,\nu>0),\,\delta\leq\min\{\mu,\nu\}, \\
              \quad\quad0, & \quad\gamma=0{\,}(\mu=0,{\,}\nu=\alpha\geq\delta),
            \end{array}
          \right.
\end{align*}
or
\begin{align*}
-1<b<a\quad{\rm and }\quad\left\{
            \begin{array}{ll}
              b\in\left[0,\left(5-\dfrac{\delta}{\mu}-\dfrac{\delta}{\nu}\right)\right],\quad
\gamma>0{\,}(\mu>0,\nu>0),\,\delta\leq\min\{\mu,\nu\}, \\
              \quad\quad\quad\gamma=0{\,}(\mu=0,{\,}\nu=\alpha\geq\delta).
            \end{array}
          \right.
\end{align*}
\end{theorem}
\begin{proof}
Using $\lambda(t)$ defined by \eqref{ponn-oper}, we have

\begin{align*}
\dfrac{t\lambda'(t)}{\lambda(t)}= \begin{cases} \dfrac{\left({\,}a-bt^{b-a}{\,}\right)}{\left(1-t^{b-a}{\,}\right)},
 &b\neq a,\\
a-\dfrac{1}{\log(1/t)},  &  b=a.
\end{cases}
\end{align*}

Considering both the possibilities, the proof can be divided
into the following cases.

{Case(i):} At first, let $a=b>-1$ and $\gamma>0(\mu>0,\nu>0)$.
Substituting the values of $t\lambda'(t)/\lambda(t)$ in
\eqref{eq-Gener:Starlike-Main:Cond-gamma>0} and on further
simplification gives

\begin{align}\label{eq-pons1}
a-\dfrac{1}{\log(1/t)}\leq5-\dfrac{\delta}{\mu}-\dfrac{\delta}{\nu}.
\end{align}
For $t\in(0,1)$, the function $\log(1/t)$ is positive. Thus the
condition \eqref{eq-pons1} is valid, when

\begin{align*}
a\leq5-\dfrac{\delta}{\mu}-\dfrac{\delta}{\nu} \quad{\rm and }\quad 1\leq\delta\leq\min\{\mu,\nu\}.
\end{align*}

Now, consider $a=b>-1$ and $\gamma=0(\mu=0,\nu=\alpha\geq0)$.
For the function $\lambda(t)=(a+1)^2t^a\log(1/t)$ implies
$\lambda(1)=0$. An easy computation gives

{\small{
\begin{align*}
\lambda'(t)=(a\!+\!1)^2t^{a-1}\left(a\log(1/t)\!-\!1\right)\quad {\rm and}\quad
\lambda''(t)=(a\!+\!1)^2t^{a-2}\left(a(a\!-\!1)\log(1/t)-(2a\!-\!1)\right).
\end{align*}}}
Therefore, for $a\leq0$, it is easy to see that
$\lambda(t)\geq0$, $\lambda'(t)\leq0$ and $\lambda''(t)\geq0$,
which implies that the condition
\eqref{eq-Gener:Starlike-Main-gamma0:lambda0} is true. Hence by
given hypothesis and Theorem \ref{Thm-Gener:starlike-Decr}, the
result directly follows.

{Case(ii):} Consider $a\neq b$. At first, let  $-1<b<a$ and
$\gamma>0$. Substituting the values of
$t\lambda'(t)/\lambda(t)$ in
\eqref{eq-Gener:Starlike-Main:Cond-gamma>0} is equivalent to
the inequality $\phi_t(a)\leq\phi_t(b)$, $t\in(0,1)$, where

\begin{align*}
\phi_t(b):=bt^b-\left(5-\dfrac{\delta}{\mu}-\dfrac{\delta}{\nu}\right)t^b.
\end{align*}
Now, it is required to claim that for $b\in
[0,{\,}(5-{\delta}/{\mu}-{\delta}/{\nu})]$, $\phi_t(b)$ is
decreasing function of $b$. Differentiating $\phi_t(b)$ with
respect to $b$ gives

\begin{align*}
\phi_t'(b):=bt^{b-1}\left(b-\left(5-\dfrac{\delta}{\mu}-\dfrac{\delta}{\nu}\right)\right).
\end{align*}
If $b\leq0$, then for $1\leq\delta\leq\mu$ and
$1\leq\delta\leq\mu$, clearly implies that
$\left(b-\left(5-{\delta}/{\mu}-{\delta}/{\nu}\right)\right)\leq0$
which means $\phi_t'(b)\geq0$. But, it requires to prove that
$\phi_t'(b)\leq0$, therefore we will consider the case only
when $b\geq0$. Thus $\phi_t'(b)\leq0$ is true for
$b\leq\left(5-{\delta}/{\mu}-{\delta}/{\nu}\right)$. Hence the
desired conclusion follows by the given hypothesis and Theorem
\ref{Thm-gamma>0-Generalized-Starlike}.

Now, consider $a\neq b$ and $\gamma=0$. For the function
$\lambda(t)=(a+1)(b+1){t^a(1-t^{b-a})}/{b-a}$, clearly implies
$\lambda(1)=0$. An easy computation gives

{\small{
\begin{align*}
\lambda'(t)=(a\!+\!1)(b\!+\!1)\frac{t^{a-1}(a-bt^{b-a})}{b-a}\quad {\rm and}\quad
\lambda''(t)=(a\!+\!1)(b\!+\!1)\frac{t^{a-1}(a(a\!-\!1)-b(b\!-\!1)t^{b-a})}{b-a}.
\end{align*}}}
To prove the required result, it is enough to get the
inequality \eqref{eq-Gener:Starlike-Main-gamma0:lambda0}. Since
\eqref{eq-generalized-cond-gamma0} is negative when
$1\leq\delta\leq\alpha$, therefore it remains to prove that
$\lambda(t)\geq0$, $\lambda'(t)\leq0$ and $\lambda''(t)\geq0$,
which is clearly true when $-1<a<b$ or $-1<b<a$ and this
completes the proof.
\end{proof}
\begin{remark}
\begin{enumerate}[1.]
\item Consider $a=b$, then for the case $\delta=1$ and
    $\gamma>0$, {\rm Theorem
    \ref{Thm-Generalized-starlike-pons}} gives better bound
    for $a$ as compared to {\rm\cite[Theorem
    5.7]{SarikaStarlike}} but for $\gamma=0$, the bound is
    weaker {\rm(see also \cite[Theorem 5.3]{AbeerS} for the
    case $\xi=0$)}.
\item Consider $b<a$ or $a<b$, then for $\delta=1$, {\rm
    Theorem \ref{Thm-Generalized-starlike-pons}} gives
    better range for $a$ as compared to {\rm\cite[Theorem
    5.7]{SarikaStarlike}} {\rm(see also \cite[Theorem
    5.3]{AbeerS} for the case $\xi=0$)}.
\end{enumerate}
\end{remark}
Let
\begin{align}\label{ali-singh-operator}
\lambda(t)=\dfrac{(1-k)(3-k)}{2}t^{-k}(1-t^2), \quad0\leq k<1.
\end{align}
Then the corresponding integral operator
\eqref{eq-weighted-integralOperator} by taking $\lambda(t)$
given in \eqref{ali-singh-operator} is denoted by
${T}_k^\delta(f)$. This operator is new in the literature
whereas for the particular cases of this operator many
interesting results  are available. For example, when
$\delta=1$, the operator ${T}_k^1$ was considered in the work
of R. M. Ali and V. Singh \cite{Ali}.

Now, using this generalized operator the following result is
obtained as under.
\begin{theorem}
Let $\gamma\geq0$ $(\mu\geq0, \nu\geq0)$, $k\geq0$,
$\xi\in[0,1/2]$ and $\delta\geq1$. Let $\beta<1$ satisfy
\eqref{Beta-Cond-Generalized:Starlike}, where $\lambda(t)$ is
given in \eqref{ali-singh-operator}. Then for
$f(z)\in\mathcal{W}_\beta^\delta(\alpha,\gamma)$, the function
$z^{1-\delta}\left({T}_k^\delta(f)(z)\right)^\delta$ belongs to
the class $\in\mathcal{S}^\ast(\xi)$, whenever
\begin{enumerate}[{\rm (i).}]
  \item $k\geq0,\quad\quad\quad\quad\quad\quad\quad$ for
      $\quad\gamma>0$ $(\delta\leq\mu,{\,}\delta\leq\nu)$,
  \item $k\in[2/3,1]\cup[3,\infty),\quad$ for
      $\quad\gamma=0{\,}(\mu=0,{\,}\nu=\alpha(1\leq\delta\leq\alpha))$.
\end{enumerate}
\end{theorem}
\begin{proof}
Using $\lambda(t)$ defined by \eqref{ali-singh-operator} gives

\begin{align*}
\dfrac{t\lambda'(t)}{\lambda(t)}=-k-\dfrac{2t^2}{(1-t^2)}.
\end{align*}
Now the proof is divided into the following two cases.

Case(i): For $\gamma>0$. Substituting the values of
$t\lambda'(t)/\lambda(t)$ in
\eqref{eq-Gener:Starlike-Main:Cond-gamma>0} gives

\begin{align}\label{eq-generalized-starlike-ali-singh-operator}
-k-\dfrac{2t^2}{(1-t^2)}\leq 5-\dfrac{\delta}{\mu}-\dfrac{\delta}{\nu}.
\end{align}
For $k\geq-\left(5-{\delta}/{\mu}-{\delta}/{\nu}\right)$,
\eqref{eq-generalized-starlike-ali-singh-operator} is obviously
true. As we know that $1\leq\delta\leq\mu$ and
$1\leq\delta\leq\nu$, therefore for all values of $k\geq0$,
\eqref{eq-generalized-starlike-ali-singh-operator} holds.

Case(ii): Consider $\gamma=0$. The function
$\lambda(t)={(1-k)(3-k)}t^{-k}(1-t^2)/2$ implies
$\lambda(1)=0$. Now, in order to prove the result, we need to
show that the inequality
\eqref{eq-Gener:Starlike-Main-gamma0:lambda0} holds under the
given hypothesis. Differentiating $\lambda(t)$ gives

\begin{align*}
\lambda'(t)&=\frac{(1-k)(3-k)}{2}t^{-k-1}(-k-(2-k)t^2)\quad\quad{\rm and}\\
\lambda''(t)&=\frac{(1-k)(3-k)}{2}t^{-k-2}(k(k+1)-(2-k)(1-k)t^2).
\end{align*}
For $2/3\leq k\leq1$ or $k\geq3$ gives $\lambda(t)\geq0$,
$\lambda'(t)\leq0$ and $\lambda''(t)\geq0$. Thus inequality
\eqref{eq-Gener:Starlike-Main-gamma0:lambda0} is true for all
$k\in[2/3,1]\cup[3,\infty)$ and $1\leq\delta\leq\alpha$.
\end{proof}

\section{Proof of Theorem \ref{Thm-Gener:Starlike-MainResult}}\label{sec-Gener:Starlike-Main:Result-Proof}
To show that with the given conditions
\begin{align*}
z^{1-\delta}(F_\delta(z))^{\delta}\in\mathcal{S}^\ast(\xi)\Longleftrightarrow
\mathcal{N}_{\Pi_{\mu,\nu}^\delta}(h_\xi)\geq 0,
\end{align*}
 it requires to prove that the function $z^{1-\delta}(F_\delta(z))^{\delta}$ is
univalent and satisfies the order of starlikeness condition
when $\mathcal{N}_{\Pi_{\mu,\nu}^\delta}(h_\xi)\geq 0$, and its
vice versa.

Since the case $\gamma=0(\mu=0, \nu=\alpha>0)$ corresponds to
\cite[Theorem 2.1]{Aghalary}, it is enough to obtain the
condition for $\gamma>0$. Let

{\small{
\begin{align}\label{eq-generalized:starlike-H}
H(z):=(1\!-\!\alpha\!+\!2\gamma)\!\left(\frac{f}{z}\right)^\delta
\!+\!\left(\alpha\!-\!3\gamma+\gamma\left[\left(1\!-\!\frac{1}{\delta}\right)\left(\frac{zf'}{f}\right)\!+\!
\frac{1}{\delta}\left(1\!+\!\frac{zf''}{f'}\right)\right]\right)
\left(\frac{f}{z}\right)^\delta\!\left(\frac{zf'}{f}\right).
\end{align}}}
Using \eqref{eq-mu+nu} in \eqref{eq-generalized:starlike-H}
a simple computation gives
\begin{align*}
H(z)&=(1-\mu-\nu+\mu\nu)\left(\frac{f}{z}\right)^\delta
+\left(\dfrac{}{}\mu+\nu-2\mu\nu\right)\left(\frac{f}{z}\right)^\delta\left(\frac{zf'}{f}\right)\nonumber\\
&\quad+\mu\nu\left(\left(1-\frac{1}{\delta}\right)\left(\frac{zf'}{f}\right)+
\frac{1}{\delta}\left(1+\frac{zf''}{f'}\right)\right)
\left(\frac{f}{z}\right)^\delta\left(\frac{zf'}{f}\right)\notag\\
&=\dfrac{\mu\nu}{\delta^2}z^{1-\delta/\mu}\left(z^{\delta/\mu-\delta/\nu+1} \left(
z^{\delta/\nu}\left(\dfrac{f}{z}\right)^\delta\right)'\right)'.
\end{align*}
Let $G(z)\!=\!(H(z)\!-\!\beta)/(1\!-\!\beta)$, then it is easy
to see that for some $\phi\in\mathbb{R}$, ${\rm
Re}\left(e^{i\phi}{\,}G(z)\right)>0$. Now, using duality theory
\cite[p. 22]{Rus}, we may confine to the function $f(z)$ for
which $G(z)={(1+xz)}/{(1+yz)}$, where $|x|=|y|=1$. Thus

\begin{align*}
\dfrac{\mu\nu}{\delta^2}{\,} z^{1-\delta/\mu}\left(z^{\delta/\mu-\delta/\nu+1}
\left( z^{\delta/\nu}\left(\dfrac{f}{z}\right)^\delta\right)'\right)'
=(1-\beta)\dfrac{1+xz}{1+yz}+\beta,
\end{align*}
or equivalently,

\begin{align*}
\left(\frac{f(z)}{z}\right)^\delta
&=\dfrac{\delta^2}{\mu\nu z^{{\delta}/{\nu}}}\left(\int_0^z\dfrac{1}{\eta^{{\delta}/{\mu}-{\delta}/{\nu}+1}}
\left(\int_0^\eta\dfrac{1}{\omega^{1-{\delta}/{\mu}}}\left((1-\beta)\dfrac{1+x\omega}{1+y\omega}+\beta\right)
d\omega\right)d\eta\right)\\
&=\dfrac{\delta^2}{\mu\nu}\int_0^1\int_0^1\left(\beta+(1-\beta)\dfrac{1+xzrs}{1+yzrs}\right)r^{{\delta}/{\nu}-1}s^{{\delta}/{\mu}-1}dr ds.
\end{align*}
Using \eqref{eq-generalized:starlike-psi_munu}, the above
equality implies
\begin{align}\label{eq-generalized:starlike-f/z-psi:series}
\left(\frac{f(z)}{z}\right)^\delta
=\left(\beta+(1-\beta)\left(\dfrac{1+xz}{1+yz}\right)\right)\ast \psi_{\mu,\nu}^\delta(z).
\end{align}
Therefore
\begin{align*}
\left(z\left(\frac{f(z)}{z}\right)^\delta\right)^\prime
=
=\left(\beta+(1-\beta)
\left(\dfrac{1+xz}{1+yz}\right)\right)\ast \left(z\psi_{\mu,\nu}^\delta(z)\right)'.
\end{align*}
Using \eqref{eq-generalized:starlike-Phi_munu}, the above
equality can be rewritten as

\begin{align}\label{eq-gener:star-f/z:Deri-Phi:series}
\left(z\left(\frac{f(z)}{z}\right)^\delta\right)^\prime=\left(\beta+(1-\beta)
\left(\dfrac{1+xz}{1+yz}\right)\right)\ast
\Phi_{\mu,\nu}^\delta(z).
\end{align}
From the integral transform given by
\eqref{eq-weighted-integralOperator}, it is easy to see that

\begin{align*}
\left(\dfrac{F_\delta(z)}{z}\right)^\delta=\int_0^1\lambda(t)\left(\dfrac{f(tz)}{tz}\right)^\delta dt=
\int_0^1\dfrac{\lambda(t)}{1-tz}dt\ast \left(\dfrac{f(z)}{z}\right)^\delta
\end{align*}
\begin{align}\label{eq-weighted-integralOperator:F/z_delta}
\Longleftrightarrow\left(z\left(\dfrac{F_\delta(z)}{z}\right)^\delta\right)'=\int_0^1\dfrac{\lambda(t)}{1-tz}dt\ast
\left(z\left(\dfrac{f(z)}{z}\right)^\delta\right)'.
\end{align}
Using \eqref{eq-gener:star-f/z:Deri-Phi:series} in
\eqref{eq-weighted-integralOperator:F/z_delta} gives

\begin{align*}
\left(z\left(\dfrac{F_\delta(z)}{z}\right)^\delta\right)'\!=\!\int_0^1\!\!\!\dfrac{\lambda(t)}{1\!-\!tz}dt\ast
(1-\beta)\left(\dfrac{\beta}{(1\!-\!\beta)}+
\dfrac{1\!+\!xz}{1\!+\!yz}\right)\ast \Phi_{\mu,\nu}^\delta(z)
\end{align*}
%
%
\begin{align*}
\Longleftrightarrow\left(z\left(\dfrac{F_\delta(z)}{z}\right)^\delta\right)'=
(1-\beta)\left(\int_0^1\lambda(t)\Phi_{\mu,\nu}^\delta(tz)dt
+\dfrac{\beta}{(1\!-\!\beta)}\right)\ast\left(\dfrac{1\!+\!xz}{1\!+\!yz}\right).
\end{align*}
By Noshiro-Warschawski's Theorem (for details see \cite[Theorem
2.16]{DU}), the function $z^{1-\delta}(F_\delta(z))^{\delta}$
defined in the unit disk $\mathbb{D}$ belongs to $\mathcal{S}$
if $(z^{1-\delta}(F_\delta(z))^{\delta})'$ is contained in the
half plane not containing the origin.

 From the result given
in \cite[P. 23]{Rus} and by the above equality, it is easy to
note that

\begin{align*}
0\neq\left(z\left({F_\delta}/{z}\right)^\delta\right)'\Longleftrightarrow{\rm
Re}{\,}(1-\beta)\left(\int_0^1\lambda(t)\Phi_{\mu,\nu}^\delta(tz)dt
+\dfrac{\beta}{(1\!-\!\beta)}\right)>\dfrac{1}{2}
\end{align*}
\begin{align*}
\Longleftrightarrow{\rm Re}{\,}(1-\beta)\left(\int_0^1\lambda(t)\Phi_{\mu,\nu}^\delta(tz)dt+\dfrac{\beta}{(1\!-\!\beta)}-\dfrac{1}{2(1\!-\!\beta)}\right)>0.
\end{align*}
Using \eqref{Beta-Cond-Generalized:Starlike} in the above
inequality implies

\begin{align}\label{eq-Phi_g(t)}
{\rm Re}{\,}\int_0^1\lambda(t)\left(\Phi_{\mu,\nu}^\delta(tz)-\dfrac{1+g_{\mu,\nu}^\delta(t)}{2}\right)dt>0.
\end{align}
Further using \eqref{eq-generalized:starlike-Phi_munu} and
\eqref{eq-generalized:starlike-integral:g:gamma>0} in
\eqref{eq-Phi_g(t)} gives

\begin{align*}
{\rm Re}{\,}\int_0^1\!\!\lambda(t)\left(\int_0^1\!\!\!\int_0^1\!\!\dfrac{r^{\delta/\nu-1}s^{\delta/\mu-1}}{(1-rstz)^2}drds
-\int_0^1\!\!\!\int_0^1\!\!\dfrac{1-\xi(1+rst)}{(1\!-\!\xi)(1+rst)^2}r^{\delta/\nu-1}s^{\delta/\mu-1}drds\right)dt\!>\!0.
\end{align*}
Evidently ${\rm Re\,} \left(\frac{1}{1-rstz}\right)^2\geq
\frac{1}{(1+rst)^2}$ for $z\in\mathbb{D}$, directly implies
that

\begin{align*}
{\rm Re}{\,}\int_0^1\!\!\lambda(t)
&\left(\int_0^1\!\!\int_0^1\dfrac{r^{\delta/\nu-1}s^{\delta/\mu-1}}{(1-rstz)^2}drds
-\int_0^1\!\!\int_0^1\dfrac{1-\xi(1+rst)}{(1-\xi)(1+rst)^2}r^{\delta/\nu-1}s^{\delta/\mu-1}drds\right)dt\\
&\geq\dfrac{\xi}{(1-\xi)}\int_0^1\lambda(t)\left(\int_0^1\!\!\int_0^1\dfrac{(trs)r^{\delta/\nu-1}s^{\delta/\mu-1}}{(1+rst)^2}drds\right)dt>0.
\end{align*}
Clearly $f_1(r)={(trs)r^{\delta/\nu-1}}/{(1+rst)^2}$
is positive function of $r$ satisfying $0<r<<1$ whenever $0<s<1$ and $0<t<1$.
Hence the integral of $f_1(r)$ with respect to $r$ yields positive values that lie over $r$ axis, which implies that the function
$f_2(s)=s^{\delta/\mu-1}(\int_0^1 f_1(r)dr)$ is positive for $s\in(0,1)$ and hence the integral of function $f_2(s)$ is
positive above $s$-axis over $s\in(0,1)$. Since the function
$\lambda(t)$ is non-negative for $t\in(0,1)$,  the function $\int_0^1\lambda(t)(\int_0^1
f_2(s)ds)dt\geq0$. Thus, ${\rm Re}\,(z^{1-\delta}(F_\delta(z))^{\delta})'>0$ which leads to
the conclusion that the function $(z^{1-\delta}(F_\delta(z))^{\delta})$ is univalent in
$\mathbb{D}$.

In the next part of the theorem the condition of starlikeness
is discussed for $\gamma>0\,(\mu>0,\nu>0)$. From the theory of
convolution \cite[P. 94]{Rus}, it is clear that

\begin{align}\label{eq-gener:starlike-equiv1}
g\in\mathcal{S}^\ast(\xi)\Longleftrightarrow
\dfrac{1}{z}(g\ast h_\xi)(z)\neq0,
\end{align}
where $h_\xi(z)$ is as defined in
\eqref{eq-h(z)-extremal-starlike}.

For $f(z)\!\in\!\mathcal{W}_\beta^\delta(\alpha,\gamma)$, the
weighted integral operator $F_\delta$ defined in
\eqref{eq-weighted-integralOperator}, belong to the class
$\mathcal{S}^\ast_s(\zeta)$ with the conditions
$\left(1\!-\!\frac{1}{\delta}\right)\leq\zeta\leq\left(1\!-\!\frac{1}{2\delta}\right)$
and $\delta\geq1$, is equivalent of obtaining
$z\left(\frac{F_\delta}{z}\right)^\delta\in\mathcal{S}^\ast(\xi)$,
where $\xi=1-\delta+\delta\zeta$ and $0\leq\xi\leq\frac{1}{2}$,
i.e.,

\begin{align}\label{eq-gener:starlike-equiv2}
F_\delta\in\mathcal{S}^\ast_s(\zeta)\Longleftrightarrow z\left(\dfrac{F_\delta}{z}\right)^\delta\in\mathcal{S}^\ast(\xi)
\end{align}
with the above conditions. From
\eqref{eq-gener:starlike-equiv1}, it clearly follows that

\begin{align*}
z\left(\dfrac{F_\delta}{z}\right)^\delta\in\mathcal{S}^\ast(\xi)\Longleftrightarrow
0\neq\dfrac{1}{z}\left(z\left(\dfrac{F_\delta}{z}\right)^\delta\ast
h_\xi(z)\right).
\end{align*}
Using \eqref{eq-weighted-integralOperator:F/z_delta}, the above
inequality reduces to its equivalent form as

\begin{align*}
0 \neq\displaystyle\int_0^1\lambda(t)\left(\dfrac{f(tz)}{tz}\right)^\delta
dt\ast \dfrac{h_\xi(z)}{z}
=\displaystyle\int_0^1\dfrac{\lambda(t)}{1-tz}dt\ast\left(\dfrac{f(z)}{z}\right)^\delta\ast
\dfrac{h_\xi(z)}{z}.
\end{align*}
From equation \eqref{eq-generalized:starlike-f/z-psi:series},
substituting the value of $(f/z)^\delta$ in the above
inequality leads to
\begin{align*}
0
\neq (1\!-\!\beta)\left(\left[\int_0^1\!\!\lambda(t)\dfrac{h_\xi(tz)}{tz}dt+\dfrac{\beta}{(1\!-\!\beta)}\right]\ast
\psi_{\mu,\nu}^\delta(z)\ast\dfrac{1\!+\!xz}{1\!+\!yz}\right)
\end{align*}
which clearly holds if, and only if,
\begin{align*}
{\rm Re}{\,}(1\!-\!\beta)\left(\int_0^1\!\!\lambda(t)\dfrac{h_\xi(tz)}{tz}dt+\dfrac{\beta}{(1\!-\!\beta)}\right)\ast
\psi_{\mu,\nu}^\delta(z)>\dfrac{1}{2}.
\end{align*}
%
Using \eqref{Beta-Cond-Generalized:Starlike}, the above
expression gives
\begin{align*}
{\rm Re}{\,}\int_0^1\lambda(t)\left(\dfrac{h_\xi(tz)}{tz}-\dfrac{1+g_{\mu,\nu}^\delta(t)}{2}\right)dt\ast
\psi_{\mu,\nu}^\delta(z)\geq0
\end{align*}
which on further using \eqref{eq-generalized:starlike-psi_munu}
leads to
\begin{align*}
{\rm Re}{\,}\!\int_0^1\!\lambda(t)
\left(\int_0^1\int_0^1\! \dfrac{h_\xi(tzr^{\nu/\delta}s^{\mu/\delta})}{tzr^{\nu/\delta}s^{\mu/\delta}}drds-
\dfrac{1+g_{\mu,\nu}^\delta(t)}{2}\right)dt\geq0
\end{align*}
or equivalently,

\begin{align*}
{\rm Re}{\,}\!\int_0^1\!\lambda(t)
\left(\dfrac{\delta^2}{\mu\nu}\int_0^1\int_0^1\! \dfrac{h_\xi(tzuv)}{tzuv}u^{\delta/\nu-1}v^{\delta/\mu-1}dudv-
\dfrac{1+g_{\mu,\nu}^\delta(t)}{2}\right)dt\geq0.
\end{align*}
Moreover, changing the variable $tu=\omega$ gives

\begin{align*}
{\rm Re}\int_0^1\!\dfrac{\lambda(t)}{t^{\delta/\nu}}\left(\dfrac{\delta^2}{\mu\nu}\int_0^t\int_0^1
\dfrac{h_\xi(\omega zv)}{\omega zv}\omega^{\delta/\nu-1}v^{\delta/\mu-1}d\omega dv-t^{\delta/\nu}\dfrac{1+g_{\mu,\nu}^\delta(t)}{2}\right)dt\geq0.
\end{align*}
Now, integrating the above expression with respect to $t$ and
using \eqref{eq-weighted:Lambda} leads to

\begin{align*}
{\rm Re}\int_0^1\Lambda_\nu^\delta(t) \dfrac{d}{dt}\left(\dfrac{\delta^2}{\mu\nu}\int_0^t\int_0^1
\dfrac{h_\xi(\omega zv)}{\omega zv}\omega^{\delta/\nu-1}v^{\delta/\mu-1}d\omega dv-t^{\delta/\nu}\dfrac{1+g_{\mu,\nu}^\delta(t)}{2}\right)dt\geq0,
\end{align*}
which on further using
\eqref{eq-generalized:starlike-g:gamma>0} gives

\begin{align*}
{\rm Re}
\!\int_0^1\!\Lambda_\nu^\delta(t)t^{\delta/\nu-1}\!\left(\int_0^1\!\left(
\dfrac{h_\xi(tzv)}{tzv}-\dfrac{1-\xi(1+vt)}{(1-\xi)(1+vt)^2}\right)v^{\delta/\mu-1}dv\right)dt\geq0.
\end{align*}
Changing the variable $tv=\eta$ gives

\begin{align*}
{\rm Re}
\!\int_0^1\!\Lambda_\nu^\delta(t)t^{\delta/\nu-\delta/\mu-1}\!\left(\int_0^t\!\left(
\dfrac{h_\xi(\eta z)}{\eta z}-\dfrac{1-\xi(1+\eta)}{(1-\xi)(1+\eta)^2}\right)\eta^{\delta/\mu-1}d\eta\right)dt\geq0.
\end{align*}
Now, integrating with respect to $t$ and using
\eqref{eq-weighted:Pi}, finally gives

\begin{align*}
{\rm Re}
\!\int_0^1\!\Pi_{\mu,\nu}^\delta(t)\dfrac{d}{dt}\left(\int_0^t\!\left(
\dfrac{h_\xi(\eta z)}{\eta z}-\dfrac{1-\xi(1+\eta)}{(1-\xi)(1+\eta)^2}\right)\eta^{\delta/\mu-1}d\eta\right)dt\geq0
\end{align*}
or equivalently,

\begin{align*}
{\rm Re}\int_0^1\!\Pi_{\mu,\nu}^\delta(t)t^{\delta/\mu-1}\left(
\dfrac{h_\xi(tz)}{tz}-\dfrac{1-\xi(1+t)}{(1-\xi)(1+t)^2}\right)dt\geq0,
\end{align*}
which means that
$\mathcal{N}_{\Pi_{\mu,\nu}^\delta}(h_\xi)\geq0$ and this
completes the proof.

Now, to verify the sharpness let
$f(z)\in\mathcal{W}_\beta^\delta(\alpha,\gamma)$, therefore it
satisfies the differential equation

\begin{align}\label{eq-generalized-f/z-extremal}
\dfrac{\mu\nu}{\delta^2}{\,} z^{1-\delta/\mu}\left(z^{\delta/\mu-\delta/\nu+1}
\left( z^{\delta/\nu}\left(\dfrac{f}{z}\right)^\delta\right)'\right)'
=\beta+(1-\beta)\dfrac{1+z}{1-z}
\end{align}
with $\beta<1$ defined in
\eqref{Beta-Cond-Generalized:Starlike}. From
\eqref{eq-generalized-f/z-extremal}, an easy computation gives

\begin{align*}
z\left(\dfrac{f}{z}\right)^\delta=z+2(1-\beta)\sum_{n=1}^\infty \dfrac{\delta^2z^{n+1}}{(\delta+n\nu)(\delta+n\mu)}.
\end{align*}
Using the above expression,
\eqref{eq-weighted-integralOperator} gives

\begin{align}\label{eq-F(t)-series-form}
z\left(\dfrac{F_\delta(z)}{z}\right)^\delta
=z\int_0^1 \lambda(t)\left(\dfrac{f(tz)}{tz}\right)^\delta dt
=z+2(1-\beta)\sum_{n=1}^\infty\dfrac{\delta^2\tau_n z^{n+1}}{(\delta+n\nu)(\delta+n\mu)}
\end{align}
where $\tau_n=\int_0^1t^n\lambda(t) dt$.

The function $g_{\mu,\nu}^\delta$ defined in
\eqref{eq-generalized:starlike-integral:g:gamma>0} has its
series expansion as

\begin{align}\label{eq-gener:starlike-g:series}
g_{\mu,\nu}^\delta(t)=1+\dfrac{2\delta^2}{(1-\xi)}\sum_{n=1}^{\infty}\dfrac{(-1)^n(n+1-\xi)t^n}{(n\nu+\delta)(n\mu+\delta)},
\end{align}
which can further be represented in the form of generalized
hypergeometric function as

\begin{align}\label{eq-gener:starlike-g:hyper:series}
g_{\mu,\nu}^\delta(t)=2{\,}{\,}_4F_3\left(1,(2-\xi),\dfrac{\delta}{\mu},\dfrac{\delta}{\nu};{\,}
(1-\xi),\left(1+\dfrac{\delta}{\mu}\right),\left(1+\dfrac{\delta}{\nu}\right);{\,}-t\right)-1.
\end{align}
Now using \eqref{eq-gener:starlike-g:series} in
\eqref{Beta-Cond-Generalized:Starlike} gives

\begin{align}\label{eq-beta-series}
\dfrac{\beta}{(1-\beta)}
=-1-\dfrac{2\delta^2}{(1-\xi)}\sum_{n=1}^{\infty}\dfrac{(-1)^n(n+1-\xi)\tau^n}{(n\nu+\delta)(n\mu+\delta)}.
\end{align}
From \eqref{eq-F(t)-series-form}, it is easy to see that

\begin{align*}
\left(z\!\left(\!\dfrac{F_\delta(z)}{z}\!\right)^\delta\right)'=1+2(1-\beta)\sum_{n=1}^\infty\dfrac{(n+1)\delta^2\tau_n z^{n}}{(\delta+n\nu)(\delta+n\mu)},
\end{align*}
which means that

\begin{align*}
z\left.\left(z\left(\dfrac{F_\delta(z)}{z}\right)^\delta\right)'\right|_{z=-1}
=-1+&2\delta^2(1-\beta)\sum_{n=1}^\infty\dfrac{(-1)^{n+1}(n+1-\xi)\tau_n }{(\delta+n\nu)(\delta+n\mu)}\\
&+2\xi\delta^2(1-\beta)\sum_{n=1}^\infty\dfrac{(-1)^{n+1}\tau_n }{(\delta+n\nu)(\delta+n\mu)}.
\end{align*}
From \eqref{eq-beta-series}, an easy computation gives

\begin{align*}
z\left.\left(z\left(\dfrac{F_\delta(z)}{z}\right)^\delta\right)'\right|_{z=-1}
=-\xi+2\xi\delta^2(1-\beta)\sum_{n=1}^\infty\dfrac{(-1)^{n+1}\tau_n }{(\delta+n\nu)(\delta+n\mu)}.
\end{align*}
Further using \eqref{eq-F(t)-series-form}, the above expression
is equivalent to

\begin{align*}
z\left.\left(z\left(\dfrac{F_\delta(z)}{z}\right)^\delta\right)'\right|_{z=-1}
=\xi {\,}\left.z\left(\dfrac{F_\delta(z)}{z}\right)^\delta\right|_{z=-1}
\end{align*}
which clearly implies the sharpness of the result.
\qedhere
%
%



\begin{thebibliography}{20}
\bibitem{AbeerS}
R. M. Ali, A. O. Badghaish, V. Ravichandran\ and\ A. Swaminathan, Starlikeness of integral transforms and duality,
J. Math. Anal. Appl. {\bf 385} (2012), no.~2, 808--822.

\bibitem{Ali}
R. M. Ali\ and\ V. Singh, Convexity and starlikeness of functions defined by a class of integral operators.
Complex Variables Theory Appl. 26 (1995), no. 4, 299--309.

\bibitem{AliAAAThirdOrder}
R. M. Ali, S. K. Lee, K. G. Subramanian\ and\ A. Swaminathan,
A third-order differential equation and starlikeness of a double integral operator, Abstr. Appl. Anal. {\bf 2011}, Art. ID 901235, 10 pp.

\bibitem{BazilevicIntegOper}
I. E. Bazilevi\v c, On a case of integrability in quadratures of the Loewner-Kufarev equation, Mat. Sb. N.S. {\bf 37(79)} (1955), 471--476.

\bibitem{bernardi}
S. D. Bernardi, Convex and starlike univalent functions, Trans. Amer. Math. Soc. {\bf 135} (1969), 429--446.

\bibitem{CarlsonShaffer}
B. C. Carlson\ and\ D. B. Shaffer, Starlike and prestarlike hypergeometric functions, SIAM J. Math. Anal. {\bf 15} (1984), no.~4, 737--745.

\bibitem{saigo}
J. H. Choi, Y. C. Kim\ and\ M. Saigo, Geometric properties of convolution operators defined by Gaussian hypergeometric functions,
Integral Transforms Spec. Funct.
{\bf 13} (2002), no.~2, 117--130.

\bibitem{DeviPascu}
S. Devi\ and\ A. Swaminathan, Integral transforms of functions to be in a class of analytic functions using duality techniques,
J. Complex Anal. {\bf 2014}, Art. ID 473069, 10 pp.

%
\bibitem{DeviGenlConvex}
S. Devi\ and\ A. Swaminathan, Convexity of the Generalized Integral Transform using Duality Techniques, Submitted for publication.

\bibitem{DeviGenlUniv}
S. Devi\ and A. Swaminathan, Inclusion properties of Generalized Integral Transform using Duality Techniques, Submitted for publication.

\bibitem{DU}
P. L. Duren, {\it Univalent functions}, Grundlehren der Mathematischen Wissenschaften, 259, Springer, New York, 1983.

\bibitem{Aghalary}
A. Ebadian, R. Aghalary\ and\ S. Shams, Application of duality techniques to starlikeness of weighted integral transforms,
Bull. Belg. Math. Soc. Simon Stevin {\bf 17} (2010), no.~2, 275--285.

\bibitem{FourRusExtremal}
R. Fournier\ and\ S. Ruscheweyh, On two extremal problems related to univalent functions, Rocky Mountain J. Math. {\bf 24} (1994), no.~2, 529--538.

\bibitem{KimRonning}
Y. C. Kim\ and\ F. R\o nning, Integral transforms of certain subclasses of analytic functions,
J. Math. Anal. Appl. {\bf 258} (2001), no.~2, 466--489.

\bibitem{komatu}
Y. Komatu, On analytic prolongation of a family of operators, Mathematica (Cluj) {\bf 32(55)} (1990), no.~2, 141--145.

\bibitem{SuzeiniStarlike}
R. Omar, S. A. Halim and R. W. Ibrahim, Starlikeness of order for certain integral transforms using duality, Pre-print.

\bibitem{Rus}
S. Ruscheweyh, {\it Convolutions in geometric function theory}, S\'eminaire de Math\'ematiques Sup\'erieures, 83,
Presses Univ. Montr\'eal, Montreal, QC, 1982.

\bibitem{SarikaStarlike}
Sarika Verma, Sushma Gupta\ and\ Sukhjit Singh, Duality and Integral Transform of a Class of Analytic Functions,
Bull. Malaysian Math. Soc. To apppear.

\bibitem{SinghIntOperator}
R. Singh, On Bazilevi\v c functions, Proc. Amer. Math. Soc. {\bf 38} (1973), 261--271.

\end{thebibliography}
\end{document}